\documentclass[a4paper;11pt]{amsart}
\usepackage{mathrsfs}\usepackage{graphicx}
 \usepackage[arrow,matrix]{xy}
\usepackage{amsmath,amssymb,amscd,bbm,amsthm,mathrsfs}
\usepackage{color,xcolor}
\usepackage{graphicx}
\usepackage{manfnt}

\newtheorem{thm}{Theorem}[section]
\newtheorem{lem}{Lemma}[section]
\newtheorem{cor}{Corollary}[section]
\newtheorem{prop}{Proposition}[section]
\newtheorem{rem}{Remark}[section]

\theoremstyle{definition}

\begin{document}
\numberwithin{equation}{section}

 \title[Octonionic Monge-Amp\`{e}re  equation  and  octonionic  plurisubharmonic functions]{ On octonionic Monge-Amp\`{e}re  equation  and   pluripotential theory
 associated to octonionic  plurisubharmonic functions of two variables }
\author {Wei Wang }
\thanks{ Department of Mathematics, Zhejiang University, Zhejiang 310058, China, Email: wwang@zju.edu.cn;}
\subjclass[2020]{Primary  32W20; Secondary 32U15.}	

\begin{abstract} Several  aspects of pluripotential theory are generalized to octonionic  plurisubharmonic ($OPSH$) functions of two variables. We prove the
comparison principle for continuous $OPSH$  functions and the
quasicontinuity of locally   bounded    ones. An important tool is a formula of integration by parts for mixed  octonionic Monge-Amp\`{e}re  operator.  Various useful
properties of
octonionic
relative extremal functions  and   octonionic  capacity are established. The main difficulty is the non-associativity of octonions. However,  some weak form
of associativity can be used to covercome this difficulty. Another important ingredient in pluripotential theory  is the solution to    the   Dirichlet
problem   for the homogeneous octonionic
Monge-Amp\`ere equation on the unit ball, for which we
  show the  $C_{loc}^{1,1}$-regularity  by applying
  Bedford-Taylor's method.
The obstacle to do so is that   an  $OPSH$  function is usually not $OPSH$ under  automorphisms of the unit ball.   This issue can be solved by finding a weighted   transformation formula of $OPSH$  functions.
\end{abstract}
\keywords{ octonionic  plurisubharmonic functions;   octonionic Monge-Amp\`{e}re  equation; octonionic Monge-Amp\`{e}re  measure;  the  comparison principle;
quasicontinuity; octonionic  capacity; Lelong number
}
\thanks{The  author  is supported
by the National Nature Science Foundation in China (NSFC) (No. 12371082). }
 \maketitle
\section{Introduction}
 Alesker  \cite{alesker1} introduced   notions of   quaternionic plurisubharmonic functions and quaternionic Monge-Amp\`ere  operator, and proved a quaternionic version of   Chern-Levine-Nirenberg estimate
so that  the definition of  quaternionic Monge-Amp\`ere  operator can be extended to continuous
quaternionic plurisubharmonic functions.
He also \cite{alesker2} used the Baston operator $\triangle$ on complexified quaternionic K\"ahler manifold  to express the quaternionic Monge-Amp\`{e}re operator. Motivated by this and  $0$-Cauchy-Fueter complex in quaternionic analysis \cite{Wang}, 
Wan-Wang  \cite{wan-wang} introduced the first-order differential operators $d_0$ and $d_1$ acting on the quaternionic version of differential forms   and the notion of closedness of  a quaternionic positve current. The behavior of $d_0,d_1$ and $\Delta=d_0d_1$ is very similar to $\partial,\overline{\partial}$ and $\partial \overline{\partial}$ in several complex variables, and many results in the complex pluripotential theory   have been also extended to the quaternionic case (cf.
\cite{alesker4,alesker2,alesker6,Bou,DN,KS,wan20,wan5,WZ} and references therein). Recently, the quaternionic pluripotential theory has been generalized to   $m$-subharmonic functions and  $m$-Hessian  operator \cite{AAB,Liu-Wang},  and also  to    the noncommutative Heisenberg group \cite{wang21}.

  Alesker  also introduced plurisubharmonic functions in two octonionic
variables and
octonionic Monge-Amp\`{e}re  operator in \cite{alesker08}.  Recently,   Alesker-Gordon \cite{alesker24}
 established an octonionic
version of the Calabi-Yau theorem on   octonionic K\"ahler manifolds. In particular, they introduced  the    octonionic  version of  closed   positive currents.
The purpose of this paper is to generalize several aspects of     pluripotential theory to
octonionic plurisubharmonic functions.
Compared to the complex and  quaternionic  cases, the octonionic case is more complicated, because octonions is neither commutative nor   associative. Moreover,
we could not use exterior form to define currents.  However, we can use some weak form of associativity to covercome this difficulty.

Recall that  $\mathbb{O}$  has a
basis $ e_0, e_1,\cdots, e_7$ over $\mathbb{R}$, where $ e_0 = 1$ is the unit. They satisfy
\begin{equation*}
 e_p^2=-1,\qquad    e_p e_ q = -e_ qe_p,
\end{equation*}for $p,q=1,\dots,7$ and  $p \neq q$. A octonion $\mathbf{x} \in  \mathbb{O}$ can be written uniquely as
$
   \mathbf{x} =\sum_{p=0}^7  x_{  p }e_p 
$
 with $x_p\in\mathbb{ R}$, and its octonionic conjugate   is
$
   \overline{\mathbf{x}}: =x_0-\sum_{p=1}^7  x_{  p }e_p.
$
  Denote by $\mathbf{x}=(\mathbf{x}_1,\mathbf{x}_2)$ a point of $\mathbb{O}^2$, where $\mathbf{x}_\alpha= \sum_{p=0}^7  x_{\alpha  p }e_p$.
For each octonionic variable, we have a Dirac  operator   acting  on $\mathbb{O}$-valued functions:
\begin{equation}\label{eq:Dirac-l}
   \overline{\partial}_\alpha =\sum_{p=0}^7 e_{  p} \frac {\partial }{\partial x_{\alpha p}},
\end{equation} $\alpha=1,2$, which  is the generalization of the Cauchy-Riemann operator on $\mathbb{C}^n$ and $k$-Cauchy-Fueter opertor on $\mathbb{H}^n$ \cite{Wang}. Let
${\partial}_\alpha f
=\sum_{p=0}^7 \frac {\partial f}{\partial x_{\alpha p}}\overline{ e}_{  p}$. For a scalar function $u$, its
{\it octonionic Hessian} is defined as
\begin{equation}\label{eq:Hessian}
   Hess_{\mathbb{O}}(u):=\left(
 \begin{matrix}\overline\partial_{1}{\partial}_{1}u  & \overline\partial_{1}{\partial}_{2}u \\
 \overline  \partial_{2}{\partial}_{1}u  &\overline\partial_{2}{\partial}_{2}u
 \end{matrix}\right).
\end{equation} Denote by $\mathcal{ H}^2(\mathbb{O})$ the space   of all octonionic Hermitian $(2 \times 2)$-matrices
$
 \left(
 \begin{matrix}a&\mathbf{x}\\
   \overline{ \mathbf{x}}&b
 \end{matrix}\right)$, where $ a , b \in \mathbb{R} ,\mathbf{x}\in \mathbb{O} .
$
The {\it determinant}   is defined as
$
   \det \left(
 \begin{matrix}a&\mathbf{x}\\
   \overline{ \mathbf{x}}&b
 \end{matrix}\right):=ab-|\mathbf{x}|^2,
$ and  mixed determinant of of two octonionic Hermitian
matrices can be defined naturally as in \eqref{eq:mixed-determinant}. The {\it octonionic Monge-Amp\`{e}re  operator} is
$
   \det(Hess_{\mathbb{O}}(u) )
$  and the {\it mixed  octonionic Monge-Amp\`{e}re  operator} is
\begin{equation*}
   \det(Hess_{\mathbb{O}}(u),Hess_{\mathbb{O}}(v)) .
\end{equation*}

For  an open subset $\Omega\subset \mathbb{O}^2 $, a function $ u : \Omega \rightarrow [-\infty,+\infty)$ is called {\it octonionic plurisubharmonic}  ($OPSH$) if $u$ is upper semi-continuous and its restrictions to any
affine octonionic line is subharmonic. The space of
all octonionic plurisubharmonic on $\Omega$ is denoted by $    OPSH(\Omega)$.
Alesker  \cite{alesker08} established   the Chern-Levine-Nirenberg estimate,  by which he could construct  octonionic Monge-Amp\`ere  measures for continuous
$OPSH$ functions,  and found    applications to valuations.

{\bf Theorem A}   \cite[Theorem 1.1.3]{alesker08}
 {\it For any continuous $OPSH$ function  $u$ and $v$ on an open subset $\Omega\subset \mathbb{O}^2 $,  there exists a non-negative measure in $\Omega$,
denoted by $ \det(Hess_{\mathbb{O}}(u),Hess_{\mathbb{O}}(v))$,  which is uniquely characterized by the following two properties:
(i) this measure has the obvious meaning if $u,v\in C^2(\Omega)$;
(ii)
if   sequences of  continuous $OPSH$ functions $u_n$ and $v_n$ converges uniformly on compact subsets to
the function  $u$ and $v $,  then
 $ \det(Hess_{\mathbb{O}}(u_n),Hess_{\mathbb{O}}(v_n))  \rightarrow  \det(Hess_{\mathbb{O}}(u),Hess_{\mathbb{O}}(v)) $
weakly in the sense of measures.
}
\vskip 2mm
The measure  in Theorem A is called the {\it mixed   octonionic Monge-Amp\`ere  measure}.
The following   comparison principle for continuous $OPSH$  functions is a fundamental tool in the pluripotential theory.
\begin{thm}\label{thm:compare} Let $\Omega$ be a bounded   domain  and let  $u,v\in OPSH(\Omega)\cap C(\Omega)$. If $\{u<v\} \Subset  \Omega$, then   we have
  \begin{equation}\label{eq:compare0}
     \int_{\{u<v\}} \det(Hess_{\mathbb{O}}(u) )  \geq \int_{\{u<v\}}\det(Hess_{\mathbb{O}}(v) ).
  \end{equation}
\end{thm}

Given a compact subset $K$ of a domain $\Omega\subset \mathbb{O}^2$, the   {\it   (octonionic)  capacity} of the condenser $(K,\Omega)$ is defined as
\begin{equation}\label{capacity defi-K}
\begin{aligned}
C  (K):=\inf\left\{\int_{\Omega}  \det(Hess_{\mathbb{O}}(u)) : u\in\mathcal{U}^*(K,\Omega)\right\},
\end{aligned}
\end{equation}where
\begin{equation}\label{eq:U*}
\mathcal{U}^*(K,\Omega)=\left\{u\in OPSH(\Omega)\cap C(\Omega), u\vert_{K}\leq   -1 , \varliminf_{\mathbf{x}\to \partial\Omega} u(\mathbf{x})\geq  0    \right \}.
\end{equation}
The     comparison principle will be used  to establish various properties of   octonionic   capacity, e.g. if $K$ is an octonionic  regular compact subset of  an
octonionic  hyperconvex domain $\Omega$ in
$\mathbb{O}^2$, then
\begin{equation}\label{eq:capacity-K}
   C  (K)= \int_{K}  \det(Hess_{\mathbb{O}}(    \omega^*( \cdot,K,\Omega))),
\end{equation}where $ \omega^* (\cdot,K,\Omega)$  is   the    (octonionic) relative extremal function of the set $K$ in $\Omega$.

Another important ingredient in  pluripotential theory is the solution to    the   Dirichlet
problem   for the homogeneous octonionic
Monge-Ampere equation on the unit ball $B^2$.

\begin{thm} \label{thm:Dirichlet-ball}    For  $\varphi\in C^2(\partial
B^2)$, let $\Psi_{B^2,\varphi} $ be the Perron-Bremermann function for $B^2$ and
$\varphi$.  Then $ 
u=  
      \Psi_{B^2,\varphi} $   on $ B^2$ with
 $u=  \varphi$ on $\partial B^2 
$ belongs to $C_{loc}^{1,1}$, and
 is the unique solution to    the   Dirichlet
problem
\begin{equation}\label{eq:Dirichlet-problem}
  \left\{
  \begin{array}{ll} u\in  OPSH(B^2) \cap  C( \overline{ B^2}),\qquad&\\
     \det(Hess_{\mathbb{O}}(u)) = 0,\qquad & {\rm on}\quad B^2,\\
    u= \varphi,& {\rm on}\quad \partial B^2.
 \end{array}
  \right.
\end{equation}
\end{thm}

 It  is generalized to  continuous solution to    the   Dirichlet
problem   for the homogeneous octonionic
Monge-Amp\`ere equation on   general strictly octonionic pseudoconvex  domains in  Theorem \ref{thm:Dirichlet-ball-1}.

To show the  $C_{loc}^{1,1}$-regularity of the  Perron-Bremermann function, we will use the method of
  Bedford-Taylor \cite{BT} for complex Monge-Amp\`ere equation.
 The main difficulty to apply
 their method is that an  $OPSH$  function usually is not $OPSH$
  under  automorphisms of the unit ball,  unlike complex   plurisubharmonicity is preserved under biholomorphic transformations. This is similar to that the harmonicity does not preserved under conformal transformations on $\mathbb{R}^N \setminus\{0\}
  $ if  $N\geq 3$. But   multiplied by a suitable conformal factor, a harmonic function under a conformal transformation is still harmonic.
  We construct
  automorphisms of the unit ball explicitly, and find
a weighted  transformation formulae of  $OPSH$  functions  under    automorphisms (cf. Proposition \ref{prop:OPSH-Ta-u}), which allows us  to adapt  Bedford-Taylor's method to the octonionic case.

The  octonionic  capacity of an open set $U\subset\Omega $ is defined as
$
   C  (U) := \sup\{C  (K); K\subset U\}  .
$ The quasicontinuity   also holds for  locally  bounded   $OPSH$  functions.
  \begin{thm}\label{thm:quasicontinuity}
Any  locally bounded   $OPSH$  function is continuous almost everywhere with respect to octonionic  capacity, i.e., given $u\in
OPSH(\Omega)\cap L_{loc}^{\infty}(\Omega)$ and any $\epsilon>0$, there
exists an open set $U\subset \Omega$ such that $C  (U )<\epsilon$ and $u$ is continuous on $\Omega\setminus U$.
\end{thm}

In proofs of the above theorems,  an important tool is the following formula of integration by parts.
\begin{lem} \label{lem:parts} Let $\Omega$ be a domain in $\mathbb{O} ^2$ with a $C^1(\overline{\Omega})$ defining function $\varrho$. For $u\in C^2(\overline{\Omega})$ and a    closed function $\omega  \in C^1(\overline{\Omega},
\mathcal{ H}^2(\mathbb{O} ) )$,
we have
   \begin{equation}\label{eq:parts}\begin{split}
 \int _{ \Omega}  v \det(Hess_{\mathbb{O}}(u), \omega   )dV&=-\int _{ \Omega}  \det\left(\mathcal{T} (du \otimes dv ), \omega  \right) dV + \int_{\partial\Omega}
 v  \det\left( \mathcal{T} (du \otimes d \varrho),
 \omega
   \right) \frac { dS}{|\operatorname{grad}\varrho|} ,
\end{split}\end{equation}where $dV $ is the volume form on $\mathbb{O} ^2$  and $dS$ is the surface measure of the boundary $\partial\Omega$.
\end{lem}

In particular, if one of $u,v,w $ is compactly supported in $\Omega$,   then
   \begin{equation}\label{eq:exchange0}\begin{split}
 \int _{ \Omega}  u \det(Hess_{\mathbb{O}}(v) , Hess_{\mathbb{O}}(w) )dV&=  \int _{ \Omega}  v \det(Hess_{\mathbb{O}}(u) , Hess_{\mathbb{O}}(w) )dV.
\end{split}\end{equation}
This was proved by  Alesker-Gordon in \cite{alesker24}. If $v\equiv 1$, the right hand side of \eqref{eq:parts} is the only an integral over the boundary, which  is also a useful formula.

 The paper is organized as follows. In Section 2, we discuss basic facts about octonionic linear algebra,  $OPSH$  functions  and octonionic
 closed positive currents.  In Section 3,  we show the
 formula \eqref{eq:parts} of integration by parts for mixed
 octonionic Monge-Amp\`{e}re  operator, which plays the key role in the paper.
 The  comparison principle is established in Section 4.
  In Section 5, we give the fundamental solution  of the   octonionic Monge-Amp\`{e}re  equation  and introduce  the  Lelong number for a closed
  positive current. In Section 6, we show the  $C_{loc}^{1,1}$ regularity for the solution to the Dirichlet problem for the homogeneous octonionic
Monge-Ampere equation on the unit ball.
Various useful   properties of  octonionic relative extremal function  and   octonionic  capacity are established in Section 7.
The quasicontinuity     is established in Section 8. In the appendix,  we construct automorphisms of the unit ball in $
  \mathbb{O}^2 $  explicitly, which can  be extended to the boundary as continuous mappings,
not as smooth as in  the complex or  quaternionic  cases.

\section{Octonionic linear algebra, $OPSH$ functions and octonionic closed positive currents}
\subsection{Octonions and octonionic Hermitian $(2 \times 2)$-matrices}
The octonions form an $8$-dimensional algebra $\mathbb{O}$ over the reals $\mathbb{R}$,  which is neither
associative nor commutative.
An octonion $\mathbf{x}\in \mathbb{O}$ can be written as  $\mathbf{x}=\mathbf{ a}+\mathbf{ b} e_5$, where $  \mathbf{ a},\mathbf{ b} $ are two quaternionic
numbers. Namely, $\mathbf{x} $ can be identify with the
vector  $ (\mathbf{ a},\mathbf{ b} )$ in $\mathbb{H}^2$.  Then the multiplication of octonions can be written as
\begin{equation}\label{eq:octonionic}
  (\mathbf{ a},\mathbf{ b} )(  \mathbf{ c},\mathbf{ d} )= (\mathbf{ ac}-
 \overline{\mathbf{ d}}\mathbf{ b}, \mathbf{d a}+\mathbf{ b}\overline{\mathbf{ c}}  ).
\end{equation}

Although octonions
is non-associative, it still satisfy
some weak form of associativity.
\begin{lem}\label{lem:associative} \cite[Lemma 2.1.1]{alesker08}
Let $\mathbf{a},\mathbf{b},\mathbf{c} \in\mathbb{ O}$. Then,
\\(i) $Re((\mathbf{a}\mathbf{b})\mathbf{c}) = Re(\mathbf{a}(\mathbf{b}\mathbf{c}))$ (this real number will be denoted by $Re(\mathbf{abc})$).
\\
(ii)     $ \mathbb{O}$  is alternative, i.e. any subalgebra of $ \mathbb{O}$ generated by any two elements and their conjugates is
associative. It is always isomorphic either to $\mathbb{R}, \mathbb{C}$, or $\mathbb{H}$.
 \end{lem}

An {\it octonionic (right) line} in $ \mathbb{O}^2 $ is a real $8$-dimensional subspace $L \subset \mathbb{O}^2 $
  either of the form $\{(\mathbf{x}, \mathbf{a}\mathbf{x});\mathbf{x} \in\mathbb{O}\} $ for some $\mathbf{  a} \in\mathbb{O}  $ , or of the form $\{(\mathbf{0},
  \mathbf{x});\mathbf{x} \in\mathbb{O}\} $.
  The {\it mixed determinant} of two octonionic Hermitian
matrices $A = (
  a _{\overline{\alpha}\beta}  ) $ and $B = (
  b _{\overline{\alpha}\beta}  )$
 can be defined in analogy to the classical complex case as
\begin{equation}\label{eq:mixed-determinant}
   \det(A , B) =\frac  1
2\Big (a _{\overline{1}1}b_{\overline{2}2} + a _{\overline{2}2} b_{\overline{1}1} - 2\operatorname{Re}\left(a _{\overline{1}2}b_{\overline{2}1}\right)\Big ).
\end{equation}

An octonionic Hermitian $(2 \times 2)$-matrices $A\in \mathcal{ H}^2(\mathbb{O})$
 is called {\it positive definite} (resp. {\it nonnegative
definite}) if for any  column vector $\xi\in \mathbb{O}^2\setminus \{0\}$,
$
   Re(\xi^* A\xi) > 0,$  (resp. $ Re(\xi^* A\xi)  \geq  0).
$
For a positive definite (resp. nonnegative definite) matrix $A$,  one writes as usual $A > 0$
(resp. $ A \geq 0$). The following result is a version of the Sylvester criterion for octonionic Hermitian $(2 \times 2)$-matrices.
\begin{prop}  \label{prop:non-negative} \cite[Proposition 2.2.5]{alesker08}
   Let $A=\left(
 \begin{matrix}a&\mathbf{x}\\
   \overline{ \mathbf{x}}&b
 \end{matrix}\right)\in \mathcal{ H}^2(\mathbb{O})$. Then $A \geq 0$ if and only if $ a \geq 0$ and
$\det A\geq 0$.
\end{prop}

\begin{prop} \label{prop:positive}  (1) \cite[Claim 20.1]{alesker24} For a column vector $\xi\in \mathbb{O}^2 $, the octonionic Hermitian matrix
$
   \xi\otimes \xi^*=\left(\overline\xi_j{\xi}_k\right) 
$ is non-negative definite.
    \\  (2)  \cite[Proposition 20.2]{alesker24}  Fix $0 <\lambda<\Lambda< \infty $. Then there exist a natural number $N$ and unit
vectors $\zeta_1,...,\zeta_N \in \mathbb{O}^2$ satisfying that at least one of the two coordinates of each $\zeta_i$ is
real, and $0 <\lambda_*<\Lambda_*< \infty $ depending on $ \lambda,\Lambda $ only, such that every $A \in \mathcal{ H}^2(\mathbb{O})$
whose spectrum lies in $ [\lambda,\Lambda] $ can be written
\begin{equation*}
   A=\sum_{j=1}^N\beta_j  \zeta_j\otimes \zeta_j^*\quad  with\quad\beta_j \in [\lambda_*,\Lambda_*].
\end{equation*}
Moreover, the vectors $\zeta_1,...,\zeta_N$ can be chosen to contain the vectors $(1, 0), (0, 1) \in \mathbb{O}^2$.
\end{prop}
The non-negativity of $ \xi\otimes \xi^*$
follows from  Proposition \ref{prop:non-negative}.

\begin{prop}    \cite[Proposition 7.1]{alesker24}  (1) For any $A \in \mathcal{ H}^2(\mathbb{O})$ there exists a transformation in  ${\rm Spin}(9)$
which maps A to a real diagonal matrix.
(2) Let  $A , B\in \mathcal{ H}^2(\mathbb{O})$ with $A > 0$. Then there exists a transformation in ${\rm Spin}(1, 9)$
which maps $A$  to a matrix proportional to $I_2$, and $B$ to a diagonal matrix.
\end{prop}

\begin{lem} \cite[Lemma 6.7]{alesker24}
The linear map $\mathcal{T}: Sym^2_{\mathbb{R}}(\mathbb{O}^2)
\rightarrow \mathcal{ H}^2(\mathbb{O})$ given by
\begin{equation}\label{eq:T-def}
   \mathcal{T} (\xi \otimes \eta ) := \frac  1
2
(\xi\otimes \eta^*    + \eta\otimes \xi^*)
\end{equation}
is ${\rm G L}_2(\mathbb{O})$-equivariant, where the action of ${\rm G L}_2(\mathbb{O})$ on the source space is the second symmetric power of the defining
representation, and
its action on the target is the adjoint
representation.
\end{lem}

\begin{prop} \label{prop:det-positive}  \cite[Lemma  5.5]{alesker24}  (1)
If  $A, B \in \mathcal{ H}^2(\mathbb{O})$ are positive (resp. non-negative) definite then
      $
       \det(A, B) > 0 $ (resp. $\det(A, B)\geq 0$).

   (2)
     if $A > 0$ then for any $B \in \mathcal{ H}^2(\mathbb{O})$ one has
$\det(A, B)^
2 \geq \det A \cdot \det B$
and the equality holds if and only if $A$ is proportional to $B$ with a real coefficient.
\end{prop}

  An element $A \in \mathcal{ H}^2(\mathbb{O})$ is said to be {\it elementary strongly positive} if  $A=\xi \otimes \xi ^*$ for some  $\xi \in \mathbb{O}^2\setminus\{0\}$.  An
  element $A \in
  \mathcal{ H}^2(\mathbb{O})$ is called {\it strongly positive} if it belongs to the convex cone $  \mathcal{ H}_+^2(\mathbb{O})$ generated by elementary strongly
  positive elements.
  An $A \in \mathcal{ H}^2(\mathbb{O})$ is said to be {\it positive} if for any   strongly positive element  $B \in \mathcal{ H}^2(\mathbb{O})$, $ \det(A, B)\geq 0$. Then,
  span\,$  \mathcal{ H}_+^2(\mathbb{O})$= $  \mathcal{ H} ^2(\mathbb{O})$.

\subsection{$OPSH$ functions  and octonionic  closed positive currents} The   standard approximation of an  $OPSH$ function $u$  is given by $u_\epsilon:= u* \chi_\epsilon$, where $ \chi$ is smooth  with $\operatorname{supp} \chi\subset B^2, \chi\geq 0$ and $\int \chi=1$, and $ \chi_\epsilon(\cdot)=
 \frac 1{\epsilon^{16}} \chi(\frac 1\epsilon)$.

\begin{prop} \label{prop:QSH-m}
Let $\Omega$ be a domain in $\mathbb{O}^2$.  Then,
\\
(1)  $OPSH(\Omega) \subset SH(\Omega)$.\\
(2)The   standard approximation   $u_\epsilon=u\ast\chi_\epsilon$ of an  $OPSH$ function $u$ is also  $OPSH$, and satisfies $u_\epsilon\downarrow u$ as  $\epsilon\downarrow 0$.\\
(3) $au+bv \in OPSH(\Omega)$  for any $a,b\geq  0$ and $  u,v \in OPSH(\Omega)$.\\
(4) If $\gamma $ is a convex increasing function   on $\mathbb{R}$ and $u\in OPSH(\Omega)$, then $\gamma\circ u \in OPSH(\Omega)$. \\
(5) The limit of a uniformly converging or decreasing sequence of $OPSH$ functions is a  $OPSH$ function.\\
(6) The maximum of a finite number of $OPSH$ functions is an  $OPSH$ function; for a  locally uniformly bounded family $\{u_{\alpha}\}\subset OPSH$, the
regularization $u^* $ of the supremum $u  = \sup_{\alpha}u_{\alpha} $ is also an  $OPSH$ function.
\\(7)
If $D$ is an open subset  of $\Omega$, $ u \in  OPSH(\Omega)$, $ v \in  OPSH  (D)$  and $\limsup_{\mathbf{x}\rightarrow\mathbf{ x}_0} v(\mathbf{ x} ) \leq
u(\mathbf{ x}_0)$ for all $\mathbf{
x}_0\in \partial D\cap \Omega$,
then the function defined by
 \begin{equation}\label{eq:paste}
   \phi=\left\{
   \begin{array}{ll}
u, \qquad &{\rm on}\quad \Omega \setminus D,\\ \max\{u, v\}, \quad &{\rm on}\quad  D,
\end{array}
\right.
 \end{equation}
belongs to $OPSH(\Omega)$.
\end{prop}
 \begin{proof} See \cite{alesker08} for (1)-(6). (3)-(7) follows from properties of subharmonic functions on $\mathbb{R}^8$, by restricting to each octonionic
 line in $\mathbb{O}^2$.
 Decreasing  $u_\epsilon\downarrow u$ as  $\epsilon\downarrow 0$ comes from the subharmonicity of $u$ on $\mathbb{R}^{16}$.
\end{proof}

  We need both the left and right versions of   octonionic Dirac  operator acting  on $\mathbb{O}$-valued function $F$:
\begin{equation*}\begin{split}\overrightarrow{ {\partial}_{ \overline{\alpha} }}F&=\sum_{p=0}^7 e_p\frac {\partial F}{\partial x_{\alpha p}},\qquad
  F\overleftarrow{ {\partial}_{ \alpha}} =\sum_{p=0}^7\frac {\partial F}{\partial x_{\alpha p}}e_p,
\end{split}\end{equation*}where $\alpha=1,2$. Here $ \overrightarrow{ \partial _{ \overline{ \alpha }}}= \overline{\partial}_\alpha$ in \eqref{eq:Dirac-l}.
A smooth function $T : \Omega\rightarrow  \mathcal{ H}^2(\mathbb{O} ) $ is called  {\it closed
} if it satisfies
\begin{equation}\label{eq:closed}
 T_{\overline{\alpha } \beta}\overleftarrow{ {\partial}_{ \overline{\beta} }} =    \overrightarrow{ \partial _{ \overline{ \alpha }}} T_{\overline{\beta } \beta},
\end{equation} for $\alpha\neq\beta$, i.e.
$
T_{\overline{1 } 2}\overleftarrow{ {\partial}_{ \overline{2} }} =     \partial _{ \overline{ 1 }}  T_{\overline{2 } 2}$  and $
 T_{\overline{ 2 } 1}\overleftarrow{ {\partial}_{ \overline{1} }} =     \partial _{ \overline{ 2}}  T_{\overline{1 }1} .
$
For $u\in C_c^3(\Omega)$, $Hess_{\mathbb{O}}(u) $ is closed, because
\begin{equation}\label{eq:Hess-closed}
   (\partial_{\overline{\alpha}}\partial_\beta u) \overleftarrow{ {\partial}_{ \overline{\beta} }} =\partial_{\overline{\alpha}}\left(\partial_\beta
   u\overleftarrow{ {\partial}_{
   \overline{\beta} }}\right) =\partial_{\overline{\alpha}}(\partial_{\overline{\beta}}\partial_\beta u),
\end{equation}which, by Fourier transformations, comes from
\begin{equation*}
    (\xi_{\overline{\alpha}}\xi_\beta \widehat{u})  { {\xi}_{ \overline{\beta} }} =\xi_{\overline{\alpha}}\left(\xi_\beta \widehat{u }{ {\xi}_{
   \overline{\beta} }}\right) =\xi_{\overline{\alpha}}(\xi_{\overline{\beta}}\xi_\beta \widehat{u})
\end{equation*}
by the alternativity in Lemma \ref{lem:associative} (ii).

Let $  \mathcal{D}  ({\Omega},\mathcal{ H}^2(\mathbb{O} ) )$
be the space of all smooth $\mathcal{ H}^2(\mathbb{O} ) $-valued functions with compact  support.
   An element of the dual space $  \mathcal{D} ' ({\Omega},\mathcal{ H}^2(\mathbb{O} ) )$ is
called a {\it  current}.   A  current $T$ is said to be {\it positive} if we have $T(\eta)\geq 0$ for any strongly
positive  $\eta\in \mathcal{D}  ({\Omega},\mathcal{ H}^2(\mathbb{O} ) )$.   
It is called \emph{closed} if \eqref{eq:closed} holds in the sense of distributions. Let $\psi$ be a $\mathcal{ H}^2(\mathbb{O} ) $-valued  
Radon measures on $\Omega$. One can associate with $\psi$ the  current $T_\psi$ defined by
      \begin{equation}\label{eq:T-psi}
      T_\psi(\varphi)=\int_\Omega  \det(\psi, \varphi)
   \end{equation}
   for   $\varphi \in \mathcal{D}  ({\Omega},\mathcal{ H}^2(\mathbb{O} ) )$.
If a  current $T$ has a continuous extension to $ C  (\Omega, \mathcal{ H}^2(\mathbb{O} ))$,   it is called a  current of {\it order zero} or
{\it of measure type}. Namely,
for any  neighborhood $G\Subset \Omega$, there
exists a constant $K_G $ such that
$
   |T(\varphi)|\leq K_G \|\varphi\|_G
$
  for any   $\varphi\in C ({\Omega},\mathcal{ H}^2(\mathbb{O} ) ) $ with $\text{ supp}\,\varphi\subset G$, where $\|\varphi\|_G=\sum_{\alpha,\beta=1,2   }  \max_{\mathbf{ x}\in
G}|\varphi_{\overline{\alpha}\beta}(\mathbf{
x} )|$.

\begin{cor} \label{cor:measure}
   If $T \in \mathcal{D} ' ({\Omega},\mathcal{ H}^2(\mathbb{O} )) $
is a   positive current, then $T$ is a current defined by $\mathcal{ H}^2(\mathbb{O} )  $-valued measure
$
     \left( T _{\overline{\alpha}\beta}  \right)
$ via \eqref{eq:T-psi}.
\end{cor}
\begin{proof}
Since $\dim \mathcal{ H}^2(\mathbb{O} ) =10$, we can choose a basis of $  \mathcal{ H} ^2(\mathbb{O})$, say $\{h_1,\dots,h_{10}\}$, consisting of elementary strongly positive elements. Let $\{h_1',\dots,h_{10}'\}$
be the dual basis with respect to the nondegenerate bilinear form \eqref{eq:mixed-determinant}. Then $m_j(\phi)=T(\phi h_j)$ for any
$\phi \in \mathcal{D}  ({\Omega})$ defines a nonnegative distribution, and hence a Radon measure. Consequently, 
$
   T=\sum_{j=1}^{10} T_{m_j h_j'},
$ where $T_{m_j h_j'} $ is current defined by  \eqref{eq:T-psi}, since
\begin{equation*}
   \sum_{j }  T_{m_j h_j'}\left(\sum_k\phi_k h_k\right)=\sum_km_k(\phi_k)= T\left(\sum_k\phi_k h_k\right),
\end{equation*}
  for any
$\phi_k\in \mathcal{D}  ({\Omega})$.
\end{proof}

\begin{prop} \cite[Proposition 4.1.16]{alesker08} \label{prop:PSH-current}  Let $u: \Omega \rightarrow [-\infty,+\infty)$ be a function such that $u$ is not $
-\infty $. Then
$u$ is  $OPSH$ if and only if it satisfies the following conditions:
\\
(i) $u$ is upper semi-continuous;
\\
(ii) $u$ is locally integrable;
\\
(iii) $Hess_{\mathbb{O}}(u) $  
   is a positive current.
\end{prop}

 Note that  $  Hess_{\mathbb{O}}(u_\varepsilon) $ is   closed, i.e.
 $
    \left  (\partial_{\overline{\alpha}}\partial_\beta u_\varepsilon\right) \overleftarrow{ {\partial}_{ \overline{\beta} }} = \partial_{\overline{\alpha}}(
  {\partial}_{ \overline{\beta} }
    \partial_\beta u_\varepsilon)
$ by \eqref{eq:Hess-closed}. It follows that  $  Hess_{\mathbb{O}}(u) $ is   closed by the continuity of
  differentiation   on distributions as $\varepsilon\rightarrow 0$, and
$
    Hess_{\mathbb{O}}(u_\varepsilon)= Hess_{\mathbb{O}}(u)_\varepsilon
$.

\section{ Integration by parts }

Denote $\partial_{ \alpha p}:=\frac {\partial }{\partial x_{\alpha p}}$. For a $C^1$ scalar function $u$, we identify $du$ with the column vector
$
  \left (
  \begin{matrix}
     \overline{\partial}_{1}u\\\overline{\partial}_{2}u \end{matrix}
\right)
$ in $\mathbb{O}^2$. Then $\mathcal{T} (du \otimes dv )$ defined by \eqref{eq:T-def} is
\begin{equation}\label{eq:T-def}
   \mathcal{T} (du \otimes dv )= \frac  1
2
 \left(
 \begin{matrix}\overline \partial_{1}u{\partial}_{1}v  &\overline\partial_{1}u{\partial}_{2}v \\
 \overline \partial_{2}u{\partial}_{1}v  &\overline\partial_{2}u{\partial}_{2}v
 \end{matrix}\right)+\frac  1
2
 \left(
 \begin{matrix}\overline\partial_{1}v{\partial}_{1}u  &\overline\partial_{1}v{\partial}_{2}u \\
\overline  \partial_{2}v{\partial}_{1} u  &\overline\partial_{2}v{\partial}_{2}u
 \end{matrix}\right).
\end{equation}
$\mathcal{T} (du \otimes du ) $ is   elementary   strongly positive by   Proposition
\ref{prop:positive} (1).
By definition,
 \begin{equation*}
    \mathcal{T} (du \otimes dv )=  \mathcal{T} (dv \otimes du ) .
 \end{equation*}

\begin{proof}[Proof of Lemma \ref{lem:parts}]  Write $\omega=\left(
 \omega_{ \overline{\alpha}\beta}   \right) $.
Note that $\partial_{ \beta p}$ commutes $e_{p'}$. For $\alpha\neq\beta$, we have
\begin{equation*}\begin{split} \int _{ \Omega}  v \operatorname{ Re} (
\omega_{\overline\alpha\beta} &
 \overline{\partial}_{\beta}{\partial}_{\alpha}u ) dV=
 \int _{ \Omega} \sum_{p=0}^7  v\operatorname{ Re}\left(\omega_{\overline\alpha
 \beta}  {e}_{p}\partial_{\beta p}\partial_{\alpha}u\right)dV \\
 &=- \int _{ \Omega}\sum_{p=0}^7 \partial_{ \beta p} v \operatorname{ Re}\left(\omega_{\overline\alpha  \beta}   {e}_{p} \partial_{\alpha}u\right)dV -\int _{ \Omega}  v
 \operatorname{ Re}\left(\sum_{p=0}^7 \partial_{\beta
 p} \omega_{\overline\alpha\beta} {e}_{p}\partial_{\alpha}u\right)dV
 \\&\qquad + \int  _{\partial\Omega}\sum_{p=0}^7 v \operatorname{ Re}\left( \omega_{\overline\alpha\beta} {e}_{p}\partial_{\beta
 p}\varrho\,\partial_{\alpha}u\right) \frac { dS}{|\operatorname{grad}\varrho|}
 \\&= -\int _{ \Omega}   \operatorname{ Re}\left(\omega_{\overline\alpha\beta} \overline{\partial}_{ \beta  } v \partial_{\alpha}u\right)dV-\int v \operatorname{
 Re}\left( \left(\omega_{\overline\alpha\beta}
 \overleftarrow{ \partial_{\overline{\beta}} } \right ) \partial_{\alpha}u\right)dV+\int _{\partial\Omega} v \operatorname{ Re}\left(\omega_{\overline\alpha \beta}
 \overline{\partial}_{ \beta  } \varrho\,
 \partial_{\alpha}u\right) \frac { dS}{|\operatorname{grad}\varrho|}
 \\&=- \int _{ \Omega}   \operatorname{ Re}\left(\omega_{\overline\alpha\beta} \overline{\partial}_{ \beta  } v
 \partial_{\alpha}u\right)dV-\int v \operatorname{ Re}\left( \overline{ \partial}_{\alpha} \omega_{\overline\beta\beta}    \partial_{\alpha}u\right)dV+\int
 _{\partial\Omega} v
 \operatorname{ Re}\left(\omega_{\overline\alpha\beta} \overline{\partial}_{ \beta } \varrho \,\partial_{\alpha}u\right) \frac { dS}{|\operatorname{grad}\varrho|},
 \end{split}\end{equation*}
 by integration by part, the associativity of $Re( \mathbf{a}\mathbf{b} \mathbf{c}) $ in Lemma \ref{lem:associative}  
  (i)  and  the closedness \eqref{eq:closed} of $\omega$, where real valued function  $\partial_{ \beta p} v$   commutes each other
 term. Similarly,
\begin{equation*}\begin{split}
 \int _{ \Omega}  v \operatorname{ Re} (\omega_{\overline\alpha\alpha}& \overline\partial_{\beta}{\partial}_{\beta}u   )dV
=\int _{ \Omega} \sum_{p=0}^7 v \operatorname{ Re}\left( \omega_{\overline\alpha\alpha} {e}_{p}\partial_{\beta p}  {\partial}_{\beta}u\right)dV
\\&= -\int _{ \Omega}   \operatorname{ Re}\left(\omega_{\overline\alpha\alpha} \overline{\partial}_{ \beta  } v \partial_{ \beta}u\right)-\int _{ \Omega}  v
\operatorname{ Re}\left( \left(\omega_{\overline\alpha\alpha}
 \overleftarrow{ \partial_{\overline{\beta }} } \right ) \partial_{\beta}u\right)+\int _{\partial\Omega} v \operatorname{ Re}\left(\omega_{\overline\alpha\alpha}
 \overline{\partial}_{ \beta } \varrho\,
 \partial_{\beta}u\right) \frac { dS}{|\operatorname{grad}\varrho|}
 \\&=-\int _{ \Omega}   \operatorname{ Re}\left(\omega_{\overline\alpha\alpha}  \overline{ \partial}_{\beta } v {\partial}_{\beta}u\right)-\int _{ \Omega}  v
 \operatorname{ Re}\left(  \overline{ \partial}_{\beta  }
  \omega_{\overline\alpha\alpha}  {\partial}_{\beta}u\right)+ \int_{\partial\Omega}  v \operatorname{ Re}\left(\omega_{\overline\alpha\alpha}  \overline{
  \partial}_{\beta } \varrho\,
  {\partial}_{\beta}u \right)\frac { dS}{|\operatorname{grad}\varrho|},
\end{split}\end{equation*}since $ \omega_{\overline\alpha\alpha}  $ is real.
Then, their difference gives us
\begin{equation}\label{eq:summation-02}\begin{split}
 \int _{ \Omega}  v  \operatorname{ Re}\left(\omega_{\overline\alpha\alpha} \overline\partial_{\beta}{\partial}_{\beta}u  -\omega_{\overline\alpha\beta}
 \overline{\partial}_{\beta}{\partial}_{\alpha}u
\right )dV=&\int _{ \Omega}    \operatorname{ Re}\left(\omega_{\overline\alpha\beta} \overline{\partial}_{ \beta  } v \partial_{\alpha}u -
\omega_{\overline\alpha\alpha}  \overline{
 \partial}_{\beta } v {\partial}_{\beta}u\right) \\&+
 \int_\Omega v   \operatorname{ Re}\left( \overline{ \partial}_{\alpha} \omega_{\overline\beta\beta}    \partial_{\alpha}u -   \overline{ \partial}_{\beta  }
 \omega_{\overline\alpha\alpha}  {\partial}_{\beta}u\right)
 \\&+ \int_{\partial\Omega}  v  \operatorname{ Re}\left(\omega_{\overline\alpha\alpha}  \overline{ \partial}_{\beta  } \varrho\, {\partial}_{\beta}u  -
 \omega_{\overline\alpha\beta} \overline{\partial}_{ \beta } \varrho \,\partial_{\alpha}u\right) \frac { dS}{|\operatorname{grad}\varrho|}.
\end{split}\end{equation}On the other hand, note that by definition and  Lemma \ref{lem:associative} (1),
\begin{equation}\label{eq:Re}
   \operatorname{ Re}(abc)=Re(\overline{c}\overline{b }\overline{a}),\qquad \operatorname{ Re}(a )= \operatorname{ Re}( \overline{ a}) , \qquad   \operatorname{
   Re}(ab)= \operatorname{ Re}(b a),
\end{equation}for any $a,b\in \mathbb{O}$, by   which we get
$
   \operatorname{ Re}\left( \omega_{\overline1  2} \overline{\partial}_{2}{\partial}_{1}u
 \right)=  \operatorname{ Re}\left(\omega_{\overline2  1} \overline{\partial}_{1}{\partial}_{2}u
\right ).
$
Take the summation of the identity \eqref{eq:summation-02} for $(\alpha,\beta )=(1,2)$ and $(2,1) $ to get
 \begin{equation}\label{eq:summation-2}\begin{split}&
2\int _{ \Omega}  v \det(\omega ,Hess_{\mathbb{O}}(u) ) \\=&\int _{ \Omega} v Re(\omega_{\overline11} \overline\partial_{2}{\partial}_{2}u  -\omega_{\overline1  2}
\overline{\partial}_{2}{\partial}_{1}u
 )dV +\int _{ \Omega} v Re(\omega_{\overline22} \overline\partial_{1}{\partial}_{1}u  -\omega_{\overline2  1} \overline{\partial}_{1}{\partial}_{2}u
 )dV
\\=&\int _{ \Omega} Re\left(\omega_{\overline12} \overline{\partial}_{ 2  } v \partial_{1}u -
 \omega_{\overline11}  \overline{ \partial}_{2 } v {\partial}_{2}u +
  \omega_{\overline21} \overline{\partial}_{ 1  } v \partial_{2}u -   \omega_{\overline22}  \overline{ \partial}_{1 } v {\partial}_{1}u\right)
\\ &\quad  +\int_{\partial\Omega} v   Re\left(\omega_{\overline11}  \overline{ \partial}_{2 } \varrho \, {\partial}_{ 2}u
 -   \omega_{\overline12} \overline{\partial}_{ 2 } \varrho \, \partial_{1}u  +    \omega_{\overline22}  \overline{ \partial}_{1 } \varrho \, {\partial}_{1}u
 -   \omega_{\overline21} \overline{\partial}_{ 1  } \varrho\, \partial_{2}u\right) \frac { dS}{|\operatorname{grad}\varrho|}
  .
\end{split}\end{equation}In the above summation, the second terms on the R. H. S. of \eqref{eq:summation-02}
are cancelled each other.
Note that
\begin{equation*}\begin{split}
    Re\left(\omega_{\overline 2 1 }  \overline\partial_{1} v\partial_{2}u\right)&=Re\left(\overline \partial_{2}u\partial_{1} v  \omega_{\overline  1
    2 }  \right)= Re\left(\omega_{\overline  12}  \overline \partial_{2}u\partial_{1} v  \right),\\
  Re  \left(\omega_{\overline\alpha\alpha}  \overline{ \partial}_{\beta  } v {\partial}_{\beta}u\right)&
  =\omega_{\overline\alpha\alpha} Re\left( \overline{ \partial}_{\beta  } v
  {\partial}_{\beta}u \right)= Re\left(\omega_{\overline\alpha\alpha} \overline{ \partial}_{\beta } u
  {\partial}_{\beta} v\right),
 \end{split}\end{equation*}by using \eqref{eq:Re}   again.
  Therefore, we have
  \begin{equation}\label{eq:conjugate} \begin{split}Re\left( \omega_{\overline11} \overline{\partial}_{ 2  } v \partial_{2}u +\omega_{\overline22}
  \overline{\partial}_{ 1  } v \partial_{1}u\right)&=
  \omega_{\overline11}\mathcal{T} (du \otimes dv )_{\overline22} + \omega_{\overline2 2}\mathcal{T} (du \otimes dv )_{\overline1 1} ,\\
      Re\left( \omega_{\overline12} \overline{\partial}_{ 2  } v \partial_{1}u +\omega_{\overline21} \overline{\partial}_{ 1  } v \partial_{2}u\right)&= 2
      Re\left(\omega_{\overline12}\mathcal{T} (du \otimes dv )_{\overline21}\right)= 2 Re\left(\omega_{\overline2 1}\mathcal{T} (du \otimes dv )_{\overline1
      2}\right).
 \end{split}\end{equation}Substitute \eqref{eq:conjugate} into \eqref{eq:summation-2} to get the resulting identity \eqref{eq:parts}, since $ \det(A, B)= \det(B ,A  )$ by definition.
\end{proof}

 \begin{cor}\label{cor:exchange} Assume as in Lemma \ref{lem:parts}. Then,  (1)\begin{equation}\label{cor:boundary}\begin{split}
 \int _{ \Omega}    \det(Hess_{\mathbb{O}}(u),\omega  )dV& =\int_{\partial\Omega}  \det\left( \mathcal{T} (du \otimes d \varrho),  \omega  \right) \frac {
 dS}{|\operatorname{grad}\varrho|} .
\end{split}\end{equation}
(2) If one of $u,v $ is compactly supported in $\Omega$,   then \eqref{eq:exchange0} holds.
\end{cor}

\section{Octonionic  mixed   Monge-Amp\`{e}re   measure and the  comparison principle}

\subsection{A   uniform estimate}

We need the following key integral estimate (cf. e.g.
Sadullaev-Abdullaev
\cite[Theorem 16.2]{sadullaev0} for  complex plurisubharmonic   functions and \cite{Liu-Wang} for quaternionic   $m$-subharmonic  functions).
    
 \begin{prop}\label{prop:uniformly-bounded} For any compact subset $K\Subset\Omega$, the integrals
   \begin{equation}\label{eq:uniformly-bounded}
   \int_K u_1 \det(Hess_{\mathbb{O}}(u_2 ),Hess_{\mathbb{O}}(u_3 )) dV \quad {\rm and}    \quad \int_K\det\left(  \mathcal{T} (du_1   \otimes d u_1),
   Hess_{\mathbb{O}}(u_2 ) \right)dV
  \end{equation}
   for $u_1,u_2,u_3\in L_M:=\{u\in OPSH(\Omega)\cap C^2 (\Omega) : |u|\leq   M\} $ are
uniformly bounded.
\end{prop}
\begin{proof} Let $0\leq \chi\in C_0^\infty (\Omega)$ with $\chi\equiv 1 $ on $K$, supp $ \chi\subset K'\Subset \Omega$. We can assume $0\leq u_1,u_2\leq 2M$ if  replace $u_j$ by $u_j+M$. Then,
 \begin{equation*}\label{eq:uniformly-bounded0}\begin{split}
  \pm \int_K u_1 \det(Hess_{\mathbb{O}}(u_2 ),Hess_{\mathbb{O}}(u_3 )) dV &\leq M \int_\Omega \chi \det(Hess_{\mathbb{O}}(u_2 ),Hess_{\mathbb{O}}(u_3 )) dV\\
   &=M \int_\Omega u_2\det(Hess_{\mathbb{O}}(\chi ),Hess_{\mathbb{O}}(u_3 )) dV\\
   &\leq \frac {2 M^2}{\varepsilon_0 } \int_{K'}  \det(Hess_{\mathbb{O}}(|\mathbf{x}|^2 ),Hess_{\mathbb{O}}(u_3 )) dV,
 \end{split} \end{equation*}
by using \eqref{eq:exchange0} and
\begin{equation*}
 -  Hess_{\mathbb{O}}(|x|^2 )\leq \varepsilon_0  Hess_{\mathbb{O}}(\chi ) \leq Hess_{\mathbb{O}}(|x|^2 ),
\end{equation*} for some $\varepsilon_0>0$.
 Repeating the procedure, we get
\begin{equation*}\label{eq:uniformly-bounded0}\begin{split}
   \pm \int_K u_1 \det(Hess_{\mathbb{O}}(u_2 ),Hess_{\mathbb{O}}(u_3 )) dV &\leq  \frac {4 M^2}{\varepsilon_0\varepsilon_0' }   \int_{K''}  
   \det(Hess_{\mathbb{O}}(|\mathbf{x}|^2 ),Hess_{\mathbb{O}}(|\mathbf{x}|^2 )) dV 
 \end{split} \end{equation*}for some $\varepsilon_0'>0$ and $K\Subset K'\Subset \Omega$.
The uniformly boundedness of the first integral in \eqref{eq:uniformly-bounded} follows.

For the second integral, we may assume $B^2\Subset \Omega$, and
$-1\leq   u_j\leq   0$ on $B^2 $ without loss of generality. If replace $u_j$ by $\max\{u_j  ,v \}$ with
$v(\mathbf{x})=9(|\mathbf{x}|^2-\frac{5}{9}),$ then values of $u_j$'s  inside $B(\mathbf{0},
2/3)$ are unchanged, while $u_j\equiv v$ in a neighborhood of the sphere $ \partial B^2$, by $v \vert_{\partial
B^2}=4>0>u\vert_{\partial
B^2}$. But $u_j$'s may be only continuous now. Consider  the standard approximations $u _{j}^t\downarrow u $  by smooth $  OPSH $ functions in   a neighborhood of the ball $  B^2$. 
Noting that  $u _{j}^t\rightarrow u_j $ in $C^2( B')$ with $ B'=B(\mathbf{0},
1/2)$  and   $ \mathcal{T} (du_1^t   \otimes d u_1^t) ,Hess_{\mathbb{O}}(u_2^t)\geq 0$, we have
\begin{equation*}\begin{split}0\leq \int_{B'}\det (  \mathcal{T} (du_1   \otimes d u_1),Hess_{\mathbb{O}}(u_2) )=&\lim_{t\rightarrow+\infty}\int_{B'}\det (  \mathcal{T} (du_1^t \otimes d u_1^t),Hess_{\mathbb{O}}(u_2^t) )\\
\leq&\lim_{t\rightarrow+\infty}\int_{B^2 }\det\left(  \mathcal{T} (du_1 ^t  \otimes d
u_1^t), Hess_{\mathbb{O}}(u_2^t) \right)\\= &-
 \lim_{t\rightarrow+\infty}\int _{B^2 } u_1^t \det(Hess_{\mathbb{O}}(u_1 ^t),Hess_{\mathbb{O}}(u_2 ^t)) \\&  + \int_{\partial B^2}  v \det\left(
   \mathcal{T} (dv  \otimes d \varrho) , Hess_{\mathbb{O}}(v ) \right) \frac { dS}{|\operatorname{grad}\varrho|} ,
\end{split}
\end{equation*}by applying integration by parts in Lemma \ref{lem:parts} and $u_1^t,u_2^t=v^t$ on the boundary for large $t$. We have proved the first integral over  $ B^2$ on the
R. H. S.  uniformly  bounded, while the second
integral over  $\partial B^2$ is bounded by an absolute constant, which only depends on  $v,\varrho$    on the boundary. The result follows.
 \end{proof}

 The following
  proposition  follows from    Proposition  \ref{prop:uniformly-bounded} and Theorem A.

\begin{cor}\label{cor:uniformly-bounded-C}
In the function class $L_M=\{u\in OPSH(\Omega)\cap C (\Omega) : |u|\leq   M\},$ the  families of  measures $  \det(Hess_{\mathbb{O}}(u),Hess_{\mathbb{O}}(v)) $ are
locally uniformly  bounded.
\end{cor}

\subsection{The  comparison principle }
 We need the following proposition to prove the  comparison principle.
\begin{prop}\label{prop:compare}
 Let $\Omega$ be a bounded   domain with a $C^2$ defining function $\varrho$, and let $u ,v\in C^2(\overline{\Omega})\cap OPSH(\Omega)$. If $u=v$ on $\partial\Omega$ and
 $u\leq  v$ in $\Omega$, then
 \begin{equation}\label{eq:compare}
   \int_\Omega \det(Hess_{\mathbb{O}}(u) )\geq  \int_\Omega\det(Hess_{\mathbb{O}}(v) ).
 \end{equation}
\end{prop}
\begin{proof}   Since
\begin{equation}\label{eq:difference}\begin{split}
    \det(Hess_{\mathbb{O}}(u_1),Hess_{\mathbb{O}}(u_2))- \det(Hess_{\mathbb{O}}( v_1),Hess_{\mathbb{O}}(v_2)) = &
    \det(Hess_{\mathbb{O}}(u_1-v_1),Hess_{\mathbb{O}}(u_2)) \\
    &+\det(Hess_{\mathbb{O}}( v_1),Hess_{\mathbb{O}}(u_2-v_2))  ,
\end{split}\end{equation}by the linearity \eqref{eq:mixed-determinant} of the mixed determinant, we get
\begin{equation}\label{eq:boundary-term}\begin{split}  \int_\Omega\det(Hess_{\mathbb{O}}(u) )-\det(Hess_{\mathbb{O}}(v) ) = &
\int_\Omega \det(Hess_{\mathbb{O}}(u-v),Hess_{\mathbb{O}}(u))\\  & +\int_\Omega \det(Hess_{\mathbb{O}}( v),Hess_{\mathbb{O}}(u-v))
  \\=& \int_{\partial\Omega } \det (\mathcal{T} (d(u-v) \otimes d\varrho ) ,Hess_{\mathbb{O}}(u))\frac { dS}{|\operatorname{grad}\varrho|}
  \\ &  + \int_{\partial\Omega } \det(Hess_{\mathbb{O}}(v),  \mathcal{T} (d(u-v) \otimes d\varrho ) )\frac { dS}{|\operatorname{grad}\varrho|}
\end{split}\end{equation}
  by   Corollary \ref{cor:exchange} (1).  Since $u\leq v$ in $\Omega$, then for a point $\mathbf{x}\in \partial\Omega$ with ${\rm grad} (u -v)(\mathbf{x})\neq 0$,
  we can
  write  $u-v= h \varrho$ on a neighborhood of $\mathbf{x}$ for some   smooth function $h\geq0$. Consequently, we have $\partial_{\alpha p} (u -v) =
  h  \partial_{\alpha p} \varrho$  on
  $\partial\Omega$, and so
\begin{equation}\label{eq:boundary-term2}\begin{split}
  \mathcal{T} (d(u-v) \otimes d\varrho )   = h  \mathcal{T} (d\varrho \otimes d\varrho ),
\end{split}\end{equation}on $\partial\Omega$  near $\mathbf{x}$, 
which is elementary   strongly positive by Proposition \ref{prop:positive}.
 Since $Hess_{\mathbb{O}}(u)$ is   also   nonnegative, we see that
 $\det (\mathcal{T} (d(u-v) \otimes d\varrho ) ,Hess_{\mathbb{O}}(u))\geq0$   by definition. So the integrands in the R. H. S. of   (\ref{eq:boundary-term}) on
 $\partial\Omega$  is nonnegative if ${\rm grad} (u-v)(\mathbf{x})\neq 0$. While if ${\rm grad} (u-v)(\mathbf{x})= 0$,   they  vanish. Therefore,   (\ref{eq:boundary-term}) is  nonnegative.
\end{proof}

{\it Proof of Theorem \ref{thm:compare}}. At first, we assume that $u,v\in OPSH(\Omega)\cap C^\infty(\Omega)$. Then 
  \begin{equation*}
    G:=\{\mathbf{x}\in \Omega;u(\mathbf{x}) < v (\mathbf{x}) \} = \cup_{\eta>0}G _\eta ,
 \end{equation*}where  $G_\eta
 := \{\mathbf{x}\in \Omega;u(\mathbf{x}) < v (\mathbf{x})-\eta\}$.
 By Sard's theorem,
 $G_\eta$
 are open sets with smooth boundaries
for almost all $\eta>0$. Namely,  $G_\eta$ for such a $\eta$ has $C^\infty$ defining function $u-v+\eta$. For such $\eta$, we have
\begin{equation*}
   \int_{ G_\eta }\det(Hess_{\mathbb{O}}(u) )  \geq \int_{G_\eta}\det(Hess_{\mathbb{O}}(v) ),
\end{equation*}
by Proposition  \ref{prop:compare}. By taking limit $\eta\rightarrow 0$, we get
\begin{equation}\label{eq:compare-G-eta}
   \int_{ G  }\det(Hess_{\mathbb{O}}(u) )  \geq \int_{G }\det(Hess_{\mathbb{O}}(v) ).
\end{equation}

Now if $u,v\in OPSH(\Omega)\cap C (\Omega)$, consider  the standard approximations $u _{j}\downarrow u $, $v _{j}\downarrow v $ by smooth $  OPSH $ functions in an open subset $E$ such that $G\Subset E\Subset \Omega$.
Denote
\begin{equation*}
   G_p:=\{ \mathbf{x}\in
G; u(\mathbf{x}) < v(\mathbf{x}) -1/p\} \quad {\rm and} \quad G_{j,k,p}:=\{\mathbf{x}\in G; u_{j}(\mathbf{x}) < v_{k}(\mathbf{x}) -1/p\} .
\end{equation*}
For any open set $G'\Subset  G$ we can choose positive integers $p_0$ and $p_1$ such
that $G' \Subset  G _{p_0}
 \Subset  G_{p_1}
 \Subset  G$. Since  $u _{j}$ and $ v _{j} $
 converge locally uniformly in $G$,
there exist   $k_0$
 such that $G'\Subset G_{j,k,{p_0}}\subset  G_{p_1}\Subset  G$
 for all $j,k >k_0$. Then
 \begin{equation*}
   \int_{ G_{j,k,{p_0}}} \det(Hess_{\mathbb{O}}(u_j) )   \geq \int_{G_{j,k,{p_0}}} \det(Hess_{\mathbb{O}}(v_k ) )
\end{equation*}
 for all $j,k >k_0$ by using \eqref{eq:compare-G-eta}. Consequently,
  \begin{equation*}
   \int_{  {G}_{p_1}}\det(Hess_{\mathbb{O}}(u_j) )  \geq\int_{ {G}'}\det(Hess_{\mathbb{O}}(v_k ) )  .
\end{equation*}
By  convergence of   octonionic Monge-Amp\`ere  measures in Theorem A, we get
 \begin{equation*}
   \int_{ {G }}\det(Hess_{\mathbb{O}}(u) )  \geq \int_{ {G}_{p_1}}\det(Hess_{\mathbb{O}}(u) )  \geq \int_{G' }         \det(Hess_{\mathbb{O}}(v) ).
\end{equation*}The result follows since the  $G'\Subset  G$ is arbitrarily chosen.
\qed

\begin{prop}Let $\Omega$ be a bounded   domain with smooth boundary  and let $u ,v\in C (\overline{\Omega})\cap OPSH(\Omega)$. Suppose that
$\det(Hess_{\mathbb{O}}(u) )  \leq \det(Hess_{\mathbb{O}}(v) )$ on $\Omega$, and
  $\underline{\lim}_{\mathbf{x}\in \Omega}(u(\mathbf{x})- v(\mathbf{x})) \geq 0$. Then $u \geq v $ in $\Omega$.
\end{prop}
\begin{proof} Assume that $v(\mathbf{x}_0)-u(\mathbf{x}_0) =\eta > 0$ at
some point $\mathbf{x}_0\in \Omega$. Thus the open set $G:=\{\mathbf{x}\in\Omega:   u(\mathbf{x} )< v(\mathbf{x}) - \eta /4 \}$
is not empty and relatively compact in $\Omega$. Then for sufficiently small $\varepsilon>0$, we have
\begin{equation*}
   G_1:=\{ D:   u(\mathbf{x} )< v(\mathbf{x}) - \eta /2 +\varepsilon|\mathbf{x}-\mathbf{x}_0|^2\}\Subset G,
\end{equation*}
    and it obviously contains $\mathbf{x}_0$. By applying the  comparison principle  in Theorem \ref{thm:compare}, we get
\begin{equation*}\begin{split}
   \int_{ G_1}\det(Hess_{\mathbb{O}}(u ) )  & \geq  \int_{ G_1}  \det(Hess_{\mathbb{O}}(  v  +\varepsilon |\cdot-\mathbf{x}_0|^2 ) )  \\&  \geq \int_{
   G_1}\det(Hess_{\mathbb{O}}(  v    ) )  +\varepsilon^2  \int_{ G_1} \det(Hess_{\mathbb{O}}( |\cdot-\mathbf{x}_0|^2 ) ) ,
\end{split}\end{equation*}where the second inequality follows from the  identity \eqref{eq:difference}.
This contradicts to the assumption that $\det(Hess_{\mathbb{O}}(u) )  \leq \det(Hess_{\mathbb{O}}(v) )$.
\end{proof}

We also need the following  proposition for two $OPSH$ functions.
\begin{cor} \label{cor:compare}
Let $\Omega$ be a bounded   domain and let $u_j ,v_j\in C ( {\Omega})\cap OPSH(\Omega)$. If $u_j=v_j$ outside a compact subset of  $ \Omega$, then
 \begin{equation}\label{eq:compare=}
    \int_\Omega  \det(Hess_{\mathbb{O}}(u_1),Hess_{\mathbb{O}}(u_2))=  \int_\Omega \det(Hess_{\mathbb{O}}(v_1),Hess_{\mathbb{O}}(v_2))  .
 \end{equation}
\end{cor} \begin{proof}  If the domain has smooth boundary and $u_j ,v_j\in C^2(\overline{\Omega})\cap OPSH(\Omega)$, this identity
is obtained as in \eqref{eq:boundary-term} by applying \eqref{eq:difference},
since the  boundary terms vanish. The general case easily follows from approximation.
\end{proof}
  \section{The fundamental solution  of the   octonionic Monge-Amp\`{e}re  equation  and the  Lelong number for a closed positive current}

\begin{prop} \label{prop:fundamental-solution}
The function
\begin{equation*}
   K (\mathbf{x}): =-\frac{1}{ |\mathbf{x}-\mathbf{a} |^{6}}
\end{equation*}
  is $OPSH$ on $\mathbb{O}^2$ and is the fundamental solution to the  octonionic Monge-Amp\`{e}re  equation,
i.e.
  \begin{equation}
  \begin{aligned}\label{LLLL}
 \det(Hess_{\mathbb{O}}( K ) ) = C \delta _\mathbf{a},
\end{aligned}
  \end{equation} for some constant $C>0$.
   \end{prop}
\begin{proof}
Without loss of generality, we may assume that $\mathbf{a}=\mathbf{0}$. Denote
$
   K_{ \epsilon} (\mathbf{x}):=-\frac{ 1}{({|\mathbf{x}|^2+\epsilon})^{3}}.
$
  Then,
\begin{equation}
  \begin{aligned}\overline
 \partial_\alpha{\partial}_\beta K_{ \epsilon} (\mathbf{x})
  = \overline\partial_\alpha \frac{6\overline{\mathbf{x}}_\beta}{(|\mathbf{x}|^2+\epsilon)^{4}} 
= -\frac{48}{(|\mathbf{x}|^2+\epsilon)^{5}}{\mathbf{x}}_\alpha\overline{\mathbf{x}}_\beta+\frac{ 48\delta_{\alpha\beta}}{(|\mathbf{x}|^2+\epsilon)^{4} }.
\end{aligned}
  \end{equation}
 Hence,
 $
 \overline
 \partial_\alpha{\partial}_\alpha  K_{ \epsilon}=48\frac{  |\mathbf{x}_\beta|^2+\epsilon }{(|\mathbf{x}|^2+\epsilon)^{5}}> 0,
  $
where $\beta\neq\alpha$,   and
  \begin{equation*}
  \begin{aligned}
 2\det(Hess_{\mathbb{O}}( K_{ \epsilon}) )
                          =& \frac{48^2 (|\mathbf{x}_1|^2+\epsilon) (|\mathbf{x}_2|^2+\epsilon)}{(|\mathbf{x}|^2+\epsilon)^{10}}-\frac{48^2\operatorname{ Re} ( (\mathbf{x}_1\overline{\mathbf{x}}_2) (\mathbf{x}_2 \overline{\mathbf{x}}_1) )                          }
                          {(|\mathbf{x}|^2+\epsilon)^{10}} 
                          = \frac{48^2  \epsilon }{(|\mathbf{x}|^2+\epsilon)^{9}} > 0,
\end{aligned}
  \end{equation*} by the alternativity.  Thus, $Hess_{\mathbb{O}}( K_{ \epsilon})$ is positive by Proposition \ref{prop:non-negative}, and so $K_{ \epsilon}\in
  OPSH(\mathbb{O}^{2})$ by Proposition \ref{prop:PSH-current}. Then,  $K
  \in OPSH(\mathbb{O}^{2})$ by  Proposition \ref{prop:QSH-m} (5), since $K_{ \epsilon} \downarrow K $.
  Now letting $\epsilon\to 0$, we get
$
\det(Hess_{\mathbb{O}}( K ) ) =0 $ on $ \mathbb{O}^2\setminus \{0\}.
$

  To show  \eqref{LLLL}, for any $\varphi \in  {C}_0 (\mathbb{O}^{2})$,
  we get
   \begin{equation*}
   \int_{\mathbb{R}^{16}}  \frac{\epsilon}{(|\mathbf{x}|^2+\epsilon )^{9}} \varphi(\mathbf{x})dV
  =  \int_{\mathbb{R}^{16}} \frac{\varphi(\mathbf{x}'\epsilon^{\frac{1}{2}}) }{( |\mathbf{x}' |^2+1)^{9} }
dV(\mathbf{x}') \rightarrow C\varphi(0) ,
     \end{equation*}as $\epsilon\to0 $, by rescaling   $\mathbf{x}=\mathbf{x}'\epsilon^{\frac{1}{2}}$.
Thus $(\ref{LLLL})$ follows.
\end{proof}
  \begin{prop} \label{prop:real-form}
  Suppose that $\Omega \subseteq\mathbb{O}^{2}$ is a domain and $B(\mathbf{a},R)\Subset \Omega$ for some $R > 0$. For a closed positive current $\omega$ on $\Omega
  $ and   $0< r < R$,
  denote
\begin{equation}\label{eq:sigma}
  \begin{aligned}
   \sigma(\mathbf{a},r)=\int_{B(\mathbf{a},r)} \det(Hess_{\mathbb{O}}(|\mathbf{x}|^2),\omega ).
\end{aligned}
  \end{equation}
Then, $\frac{\sigma(\mathbf{a},r)} {{r}^8}$ is an increasing function of $r$ for $0< r < R$, and
\begin{equation}
v_\mathbf{a}(u)=\lim_{r\to0+} \frac{\sigma(\mathbf{a},r)}{{r}^8}
\end{equation}
exists and is
  nonnegative. It is called the octonionic  Lelong number of the closed positive current $\omega$  at $\mathbf{a}$.
\end{prop}
\begin{proof}    Without loss of generality, we may assume that $\mathbf{a}=0$. Firstly, assume  $\omega $ is a smooth closed positive
current.
Consider
  \begin{equation*}
  \begin{aligned}
v_\mathbf{a}(r_1,r_2):=\int_{r_1<|\mathbf{x}|< r_2} \det(Hess_{\mathbb{O}}(K),\omega ),
   \end{aligned}
  \end{equation*}for $0< r_1< r_2 < R$. Since on the ring  $r_1<|\mathbf{x}|< r_2$, $K$ is smooth and $dK=\frac {3d |\mathbf{x}|^2 }{ |\mathbf{x}|^8 }$,
  we have
  \begin{equation*}
 \begin{split}
 v_\mathbf{a}(r_1,r_2) &=
\int_{|\mathbf{x}|=r_2}        \det\left( \mathcal{T} (d K \otimes d |\mathbf{x}|^2),\omega \right) ~
\frac{dS}{2r_2} - \int_{|\mathbf{x}|=r_1}        \det\left( \mathcal{T} (d K \otimes d |\mathbf{x}|^2), \omega \right) ~
\frac{dS}{2r_1}  \\
 &=\frac { 3  }{2r_2^{ 9}}
\int_{|\mathbf{x}|=r_2}        \det\left( \mathcal{T} (d  |\mathbf{x}|^2 \otimes d |\mathbf{x}|^2), \omega \right) ~
dS -\frac { 3  }{2r_1^{ 9}}
\int_{|\mathbf{x}|=r_1}        \det\left( \mathcal{T} (d  |\mathbf{x}|^2 \otimes d |\mathbf{x}|^2), \omega\right) ~
dS \\
 &=\frac { 3  }{ r_2^{ 8}}
\int_{|\mathbf{x}|<r_2}        \det(Hess_{\mathbb{O}}(|\mathbf{x}|^2),\omega )
  -\frac { 3  }{ r_1^{ 8}}
\int_{|\mathbf{x}|<r_1}     \det(Hess_{\mathbb{O}}(|\mathbf{x}|^2),\omega )
   \\
&=3\left(\frac{\sigma(\mathbf{a},r_2)}{ r_2 ^ 8}-\frac{\sigma(\mathbf{a},r_1)}{ r_1 ^{8}}\right)> 0,
\end{split}
\end{equation*}by using Corollary \ref{cor:exchange} (1) repeatedly. Thus, the result holds for smooth currents.

By Corollary \ref{cor:measure},  a closed positive current $\omega$ on $\Omega $ is given by a $\mathcal{ H}^2(\mathbb{O} )  $-valued measure
$
     \omega  =    (
 \omega _{\overline{\alpha}\beta}   )
$.
By using the convergence of  smooth closed positive currents
 $
  \lim\limits_{\epsilon\to0}   \omega \ast{\chi_\epsilon}=  \lim\limits_{\epsilon\to0}   \left(
   \omega _{\overline{\alpha}\beta} \ast{\chi_\epsilon} \right) =  \omega 
$ in the sense of measures, 
 we get the result.
\end{proof}

 \section{The Dirichlet problem for the homogeneous octonionic
Monge-Amp\`ere equation  }

  On a bounded domain $\Omega$   in  $\mathbb{O}^2$, a function $u \in  OPSH(\Omega)$ is called {\it maximal}
if it satisfies the  {\it maximum principle} in the class $    OPSH(\Omega)$, i.e. for any $D\Subset\Omega$, if $v\in  OPSH(D)$ and
$\underline{\lim}_{\mathbf{x}\in \partial D}(u(\mathbf{x})- v(\mathbf{x})) \geq 0$, then $u \geq v $ in $D$.
 The {\it generalized Dirichlet problem} is that of   finding   an
upper semicontinuous function $u:\overline{\Omega} \rightarrow  \mathbb{R}  $ such that $u|_{\Omega}$ is maximal $OPHS$
and $u|_{\partial\Omega}\equiv \varphi$, for given $\varphi\in C( \partial\Omega)$.  We
 denote by $\mathscr  B(\Omega, \varphi)$ the family of all functions $u\in OPHS(\Omega)$ such that
  $
    \limsup_{  \Omega \ni\mathbf{y}\rightarrow \mathbf{x}}u(\mathbf{y})\leq \varphi(\mathbf{x}),
 $
for all $\mathbf{x}\in \partial\Omega $. The {\it Perron-Bremermann function} for $\Omega$ and
$\varphi$  is defined as
\begin{equation}\label{eq:Perron-Bremermann}
   \Psi_{\Omega,\varphi}(\mathbf{x}):=\sup\{u( \mathbf{x} );u\in \mathscr B  (\Omega, \varphi)\}.
\end{equation}
  This function for complex $PHS$ functions  was introduced by Bremermann  \cite{Bre} in analogy to the
classical Perron function used in  real  potential theory.  The continuity and maximality of the  Perron-Bremermann function can be proved as in the complex case
by Walsh \cite{Walsh}.
In fact there is a more general theorem for  $\mathbb{ G}$-plurisubharmonic function by Harvey-Lawson, which includes our case as a special case.

A domain  $\Omega\subset \mathbb{O}^2$ is called {\it strictly  (octonionic)   pseudoconvex}
if near each  boundary point $\mathbf{x}\in \partial \Omega$, there exists a $C^2$ defining function $\varrho$ such that  $ Hess_{\mathbb{O}}(\varrho)$ is positive
definite.

\begin{thm} \label{thm:Perron-Bremermann} \cite[Theorem 7.6]{Harvey} Suppose that $\Omega$ is a strictly  octonionic   pseudoconvex  domain in $\mathbb{O}^2$. For  $\varphi \in C( \Omega)$,
the  Perron-Bremermann  function $ \Psi_{\Omega,\varphi}
$ is continuous on $\overline{\Omega}$ and is the solution to the generalized Dirichlet
problem, i.e. it is   maximal.
 \end{thm}

\subsection{Weighted transformation formula of $OPSH$ functions under automorphisms of the unit ball  }To  use the method of
  Bedford-Taylor \cite{BT} for complex Monge-Amp\`ere equation,  we need the automorphisms of the unit ball $
  B^2 .
$ It is  a model for  octonionic hyperbolic space. Automorphisms of  octonionic hyperbolic space is much complicated than
  the complex or  quaternionic  cases \cite{All,SW}. They are relatively easier written down in the model of {\it octonionic  Siegel upper half space}:
  \begin{equation}\label{eq:Siegel}
  \mathcal U:=\left\{( {\mathbf{y}}_{1}, {\mathbf{y}}_{2})\in\mathbb{O}^{2}:
2\operatorname{ Re} {\mathbf{y}}_{2}-| {\mathbf{y}}_{1}|^{2}>0\right\}.
 \end{equation}

For each $\mathbf{a} \in B^2$, we can construct a diffeomorphism from $B^2$ to itself (cf. Appendix)
\begin{equation}\label{eq:T--a}
   T_{\mathbf{a}}:= C^{-1}\circ D_{ \delta _{\mathbf{a}}} \circ  \tau_{ \zeta_{\mathbf{a}}}\circ  C  
\text{ such that } 
   T_{\mathbf{a}}(\mathbf{a})=\mathbf{0},
\end{equation}
  where   \eqref{eq:T--a}, the number $\delta _{\mathbf{a}}>0$ and  the boundary point
  $ \zeta _{\mathbf{a}} \in \partial \mathcal U $ are defined in \eqref{eq:delta-a} and
 \eqref{eq:zeta-a}, respectively, and  $C$ is the   \emph{Cayley transform} $C:B^2\longrightarrow \mathcal{ U} $   defined by
\begin{equation} \label{eq:Cayley}
\quad(\mathbf{x}_{1},\mathbf{x}_{2})\longmapsto\left(\sqrt 2 \mathbf{x}_{1}(1+\mathbf{x}_{2})^{-1},
(1-\mathbf{x}_{2})(1+\mathbf{x}_{2})^{-1}
\right):=( {\mathbf{y}}_{1}, {\mathbf{y}}_{2}).
\end{equation} It is a diffeomorphism  from the ball $B^2$  to $\mathcal U$
because   $1-| {\mathbf{x}} |^{2}>0$ if and only if 
\begin{equation}\label{eq:Siegel-diff}\begin{split}2 \operatorname{ Re} {\mathbf{y}}_{2}-| {\mathbf{y}}_{1}|^{2}=& 2\operatorname{ Re}\frac {1-\mathbf{x}_{2}
+\overline{\mathbf{x}}_{2}-|\mathbf{x}_{2}|^2 }{|1+\mathbf{x}_{2}|^2} -\frac {2| {\mathbf{x}}_{1}|^{2}}{|1+\mathbf{x}_{2}|^2}
= 2\frac {1-| {\mathbf{x}} |^{2}  }{|1+\mathbf{x}_{2}|^2} >0.
\end{split} \end{equation}  $C$ has  the inverse $ C^{-1}:\mathcal U \longrightarrow B^2$   given by
\begin{equation}\label{eq:C-1}
 ({\mathbf{y}}_{1},{\mathbf{y}}_{2})\longmapsto C^{-1}( \mathbf{y}_{1}, \mathbf{y}_{2}):=  \left(\sqrt 2 {\mathbf{y}}_{1}(1+{\mathbf{y}}_{2})^{-1}, (1-{\mathbf{y}}_{2})(1+{\mathbf{y}}_{2})^{-1}
\right),
\end{equation}by $(1+{\mathbf{y}}_{2})^{-1}$ commuting $ 1-{\mathbf{y}}_{2}  $.
For $\zeta=(\zeta_1,\zeta_2)\in\partial \mathcal U $, i.e.   $2\operatorname{ Re} \zeta_2= |
 \zeta_1|^2 $, the
 left translate  $\tau_\zeta:\mathcal U\rightarrow \mathcal U$ is given by
\begin{equation}\label{eqLeft-translate}
(\mathbf{x}_1, \mathbf{x}_2)\mapsto( {\mathbf{y}}_1,  {\mathbf{y}}_2):=\left (\mathbf{x}_1+\zeta_1, \mathbf{x}_2+\zeta_2+ \overline{\zeta}_1 \mathbf{x}_1\right ),
\end{equation} and for given  positive number $\delta$, the dilations is given by
$
D_{\delta}:(\mathbf{x}_1, \mathbf{x}_2)\longrightarrow  (\delta\mathbf{x}_1, \delta^{2}\mathbf{x}_2).
$
 We have
\begin{equation}\label{eq:C0}
  C(\mathbf{0})=(0,1 ) \qquad  {\rm and}   \qquad  C^{-1}(0,1 )=\mathbf{0}.
\end{equation}

We begin with the  transformation formula of the standard Laplacian operator on $\mathbb{R}^N$ under the inversion $\iota:\mathbb{R}^N \setminus\{0\} \rightarrow
\mathbb{R}^N \setminus\{0\} $ defined by $   x\mapsto \frac x{|x|^2}$. Obviously, $\iota\circ \iota=id_{\mathbb{R}^N \setminus\{0\}}$.
\begin{prop}\label{prop:conformal-covariance} Let $\Omega$ be a domain in $\mathbb{R}^N \setminus\{0\} $ for $N\geq 3$.  Then for $x\in \iota (\Omega)$ and $f\in C^2(\Omega)$, we have
\begin{equation}\label{eq:Laplacian-inversion}
   \triangle\left(\frac 1{|x|^{N-2}}f\left(\frac x{|x|^2}\right)\right)=\frac 1{|x|^{N +2}}(\triangle f)\left(\frac x{|x|^2}\right).
\end{equation}
where $\triangle$ be the standard Laplacian
operator on $\mathbb{R}^N$.
\end{prop}
\begin{proof} Note that
\begin{equation}\label{eq:Laplacian-inversion0}
  \triangle\left(\frac 1{|x|^{N-2}}f\left(\frac x{|x|^2}\right)\right)=\frac 1{|x|^{N-2}}\triangle \left( f\left(\frac x{|x|^2}\right)\right)+ 2 \nabla\frac
  1{|x|^{N-2}}\cdot \nabla \left(f\left(\frac
  x{|x|^2}\right)\right)
\end{equation}
since   $\triangle\left(\frac 1{|x|^{N-2}}\right) =0$, where $\nabla$ is the gradient on $\mathbb{R}^N$. It is direct to see that
\begin{equation}\label{eq:Laplacian-inversion1}\begin{split}
   \triangle \left( f\left(\frac x{|x|^2}\right)\right)=&\sum_{j=1}^N\frac {\partial}{\partial x_j}\sum_{k=1}^N\left(\frac {\delta_{jk}}{|x|^{2}}-\frac {2x_{j
   }x_k}{|x|^{4}}\right)\frac {\partial f}{\partial x_k}\left(\frac x{|x|^2}\right)\\
=&\sum_{j=1}^N \sum_{k,l=1}^N\left(\frac {\delta_{jk}}{|x|^{ 2}}-\frac {2x_{j }x_k}{|x|^{4}}\right)\left(\frac {\delta_{jl}}{|x|^{ 2}}-\frac {2x_{j
}x_l}{|x|^{4}}\right)\frac {\partial^2 f}{\partial x_l\partial
x_k}\left(\frac x{|x|^2}\right)\\
&+\sum_{k=1}^N\left(-\sum_{j=1}^N\frac {2\delta_{jk}x_j}{|x|^{ 4}}-\sum_{j\neq k}\frac {2 x_k}{|x|^{4}} -\frac {4 x_k}{|x|^{4}} +\sum_{j=1}^N\frac {8x_{j
}^2x_k}{|x|^{ 6}}  \right)\frac {\partial f}{\partial
x_k}\left(\frac x{|x|^2}\right)\\
=&\frac 1{|x|^{4 }}\triangle f  \left(\frac x{|x|^4}\right)-2(N-2)\sum_{k=1}^N\frac { x_k}{|x|^{ 4}}\frac {\partial f}{\partial x_k}\left(\frac x{|x|^2}\right).
\end{split} \end{equation}But
\begin{equation}\label{eq:Laplacian-inversion2}\begin{split}
   2\nabla\frac 1{|x|^{N-2}}\cdot \nabla \left(f\left(\frac x{|x|^2}\right)\right)=&2(N-2)\sum_{j=1}^N \frac {- x_j}{|x|^{N }}\sum_{l=1}^N\left(\frac
   {\delta_{jl}}{|x|^{2}}-\frac
   {2x_{j }x_l}{|x|^{4}}\right)\frac {\partial f}{\partial x_l}\left(\frac x{|x|^2}\right)\\
   =&2(N-2) \sum_{l=1}^N\frac { x_l}{|x|^{ N+2}}\frac {\partial f}{\partial x_l}\left(\frac x{|x|^2}\right).
\end{split} \end{equation}
Substituting \eqref{eq:Laplacian-inversion1}-\eqref{eq:Laplacian-inversion2} to \eqref{eq:Laplacian-inversion0}, we get \eqref{eq:Laplacian-inversion}.
\end{proof}

\begin{cor}\label{cor:conformal-covariance}
If $f$ is  subharmonic on $\Omega $, then $\frac 1{|x|^{N-2}}f\left(\frac x{|x|^2}\right)$ is subharmonic on $\iota(\Omega)$.
\end{cor}
\begin{proof}   Proposition \ref{prop:conformal-covariance} gives the result for $f\in C^2(\Omega)$. The general case follows from smooth approximation
$f_\epsilon\downarrow f$.
\end{proof}
The following is  weighted  transformation formula of  $OPSH$  functions    under   the transformation $ T_{\mathbf{a}}$.
\begin{prop}\label{prop:OPSH-Ta-u} For given $\mathbf{a} \in B^2$ and $u\in OPSH(B^2)$,  the function
\begin{equation}\label{eq:T-a-u}
   (T_{\mathbf{a}}^*u)(\mathbf{x}):=
   \frac { 1  }{ |   \Psi_{\mathbf{a}}(\mathbf{x})   |^6 } u( T_{\mathbf{a}}(\mathbf{x})),\qquad {\rm where}\quad \Psi_{\mathbf{a}}(\mathbf{x}) = 2(1+ [T_{\mathbf{a}}(\mathbf{x} )]_2 )^{-1} (1+
   \mathbf{x}_2)   ,
\end{equation}
  is also $OPSH$ on $ B^2$.
\end{prop}
$ T_{\mathbf{a}}$ is less smooth than
  complex or quaternionic  automorphisms of the unit ball,    because of the
nonassociativity of octonions
(cf. Remark \ref{rem:smooth}). However, the following smoothness, which   will be proved in the appendix, is sufficient for our purpose.
\begin{prop} \label{prop:Psi-a}  Fix a $0<\varepsilon<1$. Then,
(1) $\Psi_{\mathbf{a}}(\mathbf{ x})$ is smooth for $\mathbf{a},\mathbf{ x}\in \overline{B}(\mathbf{0}, 1-\varepsilon)$.
\\
(2) For  $\mathbf{a}\in B(\mathbf{0}, 1-\varepsilon)$,   $\Psi_{\mathbf{a}}( \cdot)$ is
continuous on $\overline{B^2}$; There exists a constant $C_\varepsilon>0$ only depending on $\varepsilon$ such that $|\Psi_{\mathbf{a}}( \mathbf{x})|\geq C_\varepsilon$ for any $\mathbf{x}\in
B^2,\mathbf{a}\in B(\mathbf{0}, 1-\varepsilon) $.
\\
(3) For   $\mathbf{x}\in \partial B^2$,  $\Psi_{\mathbf{a}}( \mathbf{x} )$ is smooth for $\mathbf{a}\in B(\mathbf{0}, 1-\varepsilon)$ with $C^k$-norm  
bounded by a constant only depending on $\varepsilon$ and given $k\in \mathbb{N}$.
   \end{prop}

   \begin{cor}\label{cor:OPSH-Ta0} (1) $T_{\mathbf{a}} $ is a diffeomorphism from $B^2$ to itself; $T_{\mathbf{a}}(\mathbf{ x})$ is smooth for $\mathbf{a},\mathbf{
   x}\in \overline{B}(\mathbf{0}, 1-\varepsilon)$.
\\
(2) $T_{\mathbf{a}} $ and $T_{\mathbf{a}}^{-1} $ are continuous  on  $\overline{B^2}$ for fixed $\mathbf{a} \in \overline{B}(\mathbf{0}, 1-\varepsilon)$.
\\
(3) For $\mathbf{x}\in \partial B^2$,  $T_{\mathbf{a}}(\mathbf{x})$ and $T_{\mathbf{a}}^{-1}(\mathbf{x})$  are smooth in $\mathbf{a}\in B(\mathbf{0}, 1-\varepsilon) $ with uniformly bounded $C^k$-norm.
\end{cor}

To prove $T_{\mathbf{a}}^*u$ is
  (\ref{eq:T-a-u}) is $OPSH$, by definition of $T_{\mathbf{a}} $ in \eqref{eq:T--a}, we need to know the transformation formula of $OPSH$ functions under
  dilations, left translations and the Cayley transformation.
\begin{prop}\label{prop:OPSH-conformal-covariance} (1) For $\zeta\in \partial \mathcal U$ and $\delta>0$, if $u\in OPSH(\mathcal U )$, then
$u \circ    \tau_{ \zeta } , u \circ D_{ \delta }  \in OPSH(\mathcal U )$.
\\
(2)
 If $u\in OPSH(\mathcal U )$, then $|1+[C(\mathbf{x})]_2|^6u(C(\mathbf{x}))$ is $OPSH$ on $B^2$; Conversely, If $u\in OPSH( B^2)$, then
 $|1+[C^{-1}(\mathbf{x})]_2|^6u(C^{-1}(\mathbf{x}))$ is $OPSH$ on $\mathcal U$.
\end{prop}
\begin{proof} (1) By definition,   if $u\in OPSH(\mathcal U)$, then for any $\alpha \in \mathbb{O}$ and  $\mathbf{a}=(\mathbf{a}_1  ,\mathbf{a}_2)\in \mathcal U$,
$u(\mathbf{a}_1+ \mathbf{\alpha}
\mathbf{t},\mathbf{a}_2+\mathbf{t})$ and  $u(\mathbf{a}_1+
\mathbf{t},\mathbf{a}_2 )$  are  subharmonic   for   $\mathbf{t}$ near the origin $\mathbf{0} $. It is equivalent to require $u(\mathbf{a}_1+
\mathbf{t},\mathbf{a}_2+\mathbf{\alpha}\mathbf{t})$  and $u(\mathbf{a}_1,\mathbf{a}_2+
\mathbf{t} )$ to be subharmonic  for any $\alpha \in \mathbb{O}$,  $\mathbf{a} \in \Omega$ and     $\mathbf{t}$ near the origin $\mathbf{0}\in \mathbb{O}$. This is   because
\begin{equation*}\begin{split}
u(\mathbf{a}  ) &\leq \frac 1 {|D(\mathbf{0},r)|}\int_{|\mathbf{t}|\leq r}u(\mathbf{a}_1+ \mathbf{\alpha} \mathbf{t},\mathbf{a}_2+\mathbf{t})dV(\mathbf{t})\\
&=\frac 1 {|D(\mathbf{0},|\alpha|r)|}\int_{|\mathbf{t}'|\leq |\alpha|r}u(\mathbf{a}_1+   \mathbf{t}',\mathbf{a}_2+\mathbf{\alpha}^{-1}\mathbf{t}')dV(\mathbf{t}'),
 \end{split}\end{equation*}
if we  take transformation $  \mathbf{t}'=\mathbf{\alpha} \mathbf{t}$, where $D(\mathbf{0},r)$ is the   ball in $\mathbb{O}$ centered at $\mathbf{0}$  with radius $r$. It follows that for fixed $\zeta=(\zeta_1,\zeta_2)\in \partial \mathcal U$,
\begin{equation*}
  u\circ    \tau_{\zeta }|_{(\mathbf{a}_1+\mathbf{t},\mathbf{a}_2+\mathbf{\alpha}\mathbf{t})} = u\left(   \mathbf{a}_1 +\zeta_1 + \mathbf{t},\mathbf{a}_2+\zeta_2+ \overline{ \zeta}_1\mathbf{a}_1
  +\left(\mathbf{\alpha}+  \overline{\zeta}_1 \right  )\mathbf{t}\right ),
\end{equation*} is  subharmonic for   $\mathbf{t}$ near the origin $\mathbf{0} $,  and so is $ u(     \tau_{\zeta }(\mathbf{a}_1 ,\mathbf{a}_2+ \mathbf{t}))$.
Similarly,
\begin{equation*}
  u\circ     D_{ \delta } |_{ ( \mathbf{a}_1+\mathbf{t},\mathbf{a}_2+\mathbf{\alpha}\mathbf{t})} = u\left(
  \delta\mathbf{a}_1+\delta\mathbf{t},\delta^2\mathbf{a}_2+\delta^2\mathbf{\alpha}\mathbf{t} \right )
\end{equation*} and $u\circ    D_{ \delta }  |_{( \mathbf{a}_1 ,\mathbf{a}_2+ \mathbf{t})} $ are  subharmonic for   $\mathbf{t}$ near the origin $\mathbf{0} $.
Therefore, $u\circ    \tau_{ \zeta } , u\circ D_{ \delta }  \in OPSH(\mathcal{ U})$.

(2) Note that
\begin{equation}\label{eq:Cayley-line-t}\begin{split}
   u\circ C|_{(\mathbf{a}_1+ \alpha \mathbf{t},\mathbf{a}_2+\mathbf{t})}&=u\left(\sqrt 2 (\mathbf{a}_1+\alpha
   \mathbf{t})(1+\mathbf{a}_2+\mathbf{t})^{-1},(1-\mathbf{a}_2-\mathbf{t})(1+\mathbf{a}_2+\mathbf{t})^{-1}\right)
   \\&=u\left(\sqrt 2  \alpha+\beta(1+\mathbf{a}_2+\mathbf{t})^{-1},-1+2(1+\mathbf{a}_2+\mathbf{t})^{-1}\right)
 \end{split} \end{equation}
by the definition of Cayley transform \eqref{eq:Cayley} and the alternativity, where $\beta:= \sqrt 2 (\mathbf{a}_1-\alpha-\alpha  \mathbf{a}_2)$. If take
\begin{equation*}
   v(\mathbf{s}):=u\left(\sqrt 2  \alpha+ \beta\mathbf{s} ,-1+2 \mathbf{s} \right),
\end{equation*}then $v$ is subharmonic on $\mathbb{O}$ near the point $(1+\mathbf{a}_2 )^{-1}$,
since $u$ is $OPSH$   near $C(\mathbf{a})\in \mathcal{U}$   and
  $  v((1+\mathbf{a}_2 )^{-1}) = u(  C(\mathbf{a}))  $ by \eqref{eq:Cayley-line-t}.  Apply  Corollary \ref{cor:conformal-covariance}   to $v$ to see that
$|1+\mathbf{a}_2+\mathbf{t}|^{-6}v((1+\mathbf{a}_2+\mathbf{t})^{-1})$
is subharmonic. Consequently, 
  \begin{equation}\label{eq:u-C}
 |1+\mathbf{a}_2+\mathbf{t}|^{-6}  u\left(\sqrt 2  \alpha+ \beta(1+\mathbf{a}_2+\mathbf{t})^{-1},-1+2(1+\mathbf{a}_2+\mathbf{t})^{-1}\right)
 \end{equation}is subharmonic in  $\mathbf{t}$ near the origin $\mathbf{0} $. But
  \begin{equation}\label{eq:[C]2}
    1+[C(\mathbf{x})]_2=2(1+\mathbf{x}_2)^{-1},
 \end{equation}
   by definition \eqref{eq:Cayley}. \eqref{eq:Cayley-line-t}-\eqref{eq:u-C} implies that $|1+[C(\mathbf{x})]_2|^6u(C(\mathbf{x}))$ is subharmonic on the line
 $\{(\mathbf{a}_1+ \alpha \mathbf{t},\mathbf{a}_2+\mathbf{t}); \mathbf{t}\in \mathbb{O}\}$. By the same argument,  it is subharmonic on the line
 $\{(\mathbf{a}_1+ \mathbf{t}, \mathbf{a}_2  ); \mathbf{t}\in \mathbb{O}\}$.
 The result for $C^{-1}$   also holds because $C^{-1}$ has the same expression \eqref{eq:C-1} as  $C $.
 The proposition is proved.
\end{proof}

\begin{proof}[Proof of Proposition \ref{prop:OPSH-Ta-u}] For $\mathbf{a},\mathbf{x}\in  B^2 $,
 denote
\begin{equation}\label{eq:def-T-u}\begin{split}
   (C^*u)(\mathbf{x})&:=|1+[C(\mathbf{x})]_2|^6u\circ C(\mathbf{x})  ,\\
   (\tau_{\zeta_{\mathbf{a}}}^*u)(\mathbf{x})&:= u\circ \tau_{\zeta_{\mathbf{a}}}(\mathbf{x})  ,\\(D_{ \delta _{\mathbf{a}}}^*u)(\mathbf{x})&:= u \circ D_{ \delta
   _{\mathbf{a}}}(\mathbf{x})   ,
\\
   ( (C^{-1 })^*u)(\mathbf{x})&:=|1+[C^{-1 }(\mathbf{x})]_2|^6u\circ C^{-1 }(\mathbf{x}) .
\end{split}\end{equation}
They are all $OPSH$ by Proposition \ref{prop:OPSH-conformal-covariance}.
It follows definition \eqref{eq:def-T-u} and \eqref{eq:[C]2} that
\begin{equation}\label{eq:T-a-u'}\begin{split}
 (C^{ *} \circ\tau_{\zeta_{\mathbf{a}}}^*\circ  D_{\delta _{\mathbf{a}}}^* \circ (C^{-1 })^*u)(\mathbf{x})&=|1+[C(\mathbf{x})]_2|^6|1+[ T_{\mathbf{a}}(\mathbf{x}
 )]_2|^6u( T_{\mathbf{a}}(\mathbf{x}))\\
 &=\frac {2^6|1+ [T_{\mathbf{a}}(\mathbf{x} )]_2|^6}{|1+ \mathbf{x}_2|^6}u( T_{\mathbf{a}}(\mathbf{x})),
\end{split}\end{equation}is $OPSH$,
which is exactly \eqref{eq:T-a-u} up to a factor $4^6$.
\end{proof}

\begin{cor}\label{cor:OPSH-Ta} For given $\mathbf{a} \in B^2$, if $u\in OPSH(B^2)$, then
\begin{equation}\label{eq:T-1-a-u}
  \left |\Psi_{\mathbf{a}}\left (T_{\mathbf{a}}^{-1}(\mathbf{x})\right) \right|^6  u\left( T_{\mathbf{a}}^{-1}(\mathbf{x})\right)
\end{equation}
 is also $OPSH$ on $ B^2$.
\end{cor}
\begin{proof} For  $\zeta\in \partial \mathcal{U}$, the inverse of $\tau_{\zeta}$ is  $\tau_{\zeta^{-1}}$, where $\zeta^{-1}=(-\zeta_1,\overline{\zeta}_2)$ by
definition \eqref{eqLeft-translate}. By definition \eqref{eq:T--a},
$
  T_{\mathbf{a}}^{-1}:= C^{-1}\circ D_{ \delta _{\mathbf{a}}^{-1}} \circ  \tau_{  \zeta_{\mathbf{a}}^{-1} }\circ  C .
$
Similar to \eqref{eq:T-a-u'}, we see that
\begin{equation} \label{eq:T-1-a-u''}\begin{split} 
\left (C^{ *} \circ\tau_{  \zeta_{\mathbf{a}} ^{-1} }^*\circ  D_{ \delta _{\mathbf{a}}^{-1}}^* \circ (C^{-1 })^*u\right)(\mathbf{x})  =\frac {2^6|1+
[T_{\mathbf{a}}^{-1}(\mathbf{x} )]_2|^6}{|1+ \mathbf{x}_2|^6}u( T^{-1}_{\mathbf{a}}(\mathbf{x}))
\end{split}\end{equation} is $OPSH$.
By the expression of $ \Psi_{\mathbf{a}}$ in \eqref{eq:T-a-u}, we have
\begin{equation*}\Psi_{\mathbf{a}}(T^{-1}_{\mathbf{a}}(\mathbf{x}))
= 2 (1+\mathbf{x}_2)^{-1}
  \left  (1+ [T^{-1}_{\mathbf{a}}(\mathbf{x})]_{2}\right)
.
\end{equation*}
 Thus the right hand side of \eqref{eq:T-1-a-u''} is exactly \eqref{eq:T-1-a-u}.
\end{proof}
\subsection{ The regularity of the Perron-Bremermann function }
\begin{prop}  \label{prop:Lip}
  Suppose that $ \varphi \in C^2(\partial B^2)$ and $u = \Psi_{B^2,\varphi}$. Then  
  $u\in Lip( \overline{ B^2})$.
 \end{prop}
 \begin{proof}
We can extend $\varphi$ to a $C_c^2(\mathbb{O}^{2})$ function which coincides with $\varphi$
on $\partial B^2$,  also denoted by $\varphi$.  Then
\begin{equation*}
  |\varphi(\mathbf{x})-\varphi(\mathbf{x}_0)-\nabla\varphi(\mathbf{x}_0)\cdot(\mathbf{x} - \mathbf{x}_0)|\leq C|\mathbf{x} - \mathbf{x}_0|^2
\end{equation*}for $\mathbf{x} , \mathbf{x}_0\in\partial B^2$
  by Taylor's expansion, where $C=\|\varphi\|_{C^2}$. Now for each $  \mathbf{x}_0\in\partial B^2$, define a real linear function
    \begin{equation}\label{eq:v-x0}
    v_{\mathbf{x}_0}(\mathbf{x} ): =\varphi(\mathbf{x}_0)+\nabla\varphi(\mathbf{x}_0)\cdot(\mathbf{x} - \mathbf{x}_0)-2C(1-\operatorname{Re}(\mathbf{x} \cdot
    \overline{ \mathbf{x}}_0))
  \end{equation}
  for $\mathbf{x} \in\mathbb{O}^{2}$, which is obvious $OPSH$. But for $|\mathbf{x} |=|\mathbf{x}_0 |=1$,
    \begin{equation*}\begin{split}
|\mathbf{x} - \mathbf{x}_0|^2 =  &|\mathbf{x} |^2+|\mathbf{x}_0 |^2-2\operatorname{Re}(\mathbf{x} \cdot \overline{ \mathbf{x}}_0)   =2(1-\operatorname{Re}(\mathbf{x} \cdot
    \overline{ \mathbf{x}}_0)).
 \end{split} \end{equation*}We see that for $\mathbf{x} ,\mathbf{x}_0\in  B^2 $,  
  \begin{equation*}
     v_{\mathbf{x}_0}(\mathbf{x} )=\varphi(\mathbf{x}_0)+\nabla\varphi(\mathbf{x}_0)\cdot(\mathbf{x} - \mathbf{x}_0)- C |\mathbf{x} - \mathbf{x}_0|^2\leq
     \varphi(\mathbf{x}  ),
 \end{equation*}and  $v_{\mathbf{x}_0}(\mathbf{x}_0 )=\varphi(\mathbf{x}_0  )$. So $v_{\mathbf{x}_0}$ is a smooth  function  and belongs to the
   family $\mathscr{B }(B^2 , \varphi)$. Set
\begin{equation*}
  v(\mathbf{x} ):=\sup\{ v_{\mathbf{x}_0}(\mathbf{x} );\mathbf{x}_0\in\partial B^2 \}, \qquad \mathbf{x} \in \overline{B^2} .
\end{equation*}It is obviously a  Lipschitzian function $\widetilde{v} $ on $\overline{B^2} $ with Lipschitzian constant $C'>0$.
Similarly, one can construct a  Lipschitzian function $\widetilde{v} $ on $\overline{B^2} $ with Lipschitzian constant $C'$, such that $-\widetilde{v} \in
\mathscr{B }(B^2, -\varphi)$ and
$\widetilde{v} = \varphi$ on $\partial B^2$. Thus ${v}\leq u\leq \widetilde v$ on $\overline{B^2} $  by the definition of  Perron-Bremermann function, and so we
get the boundary estimate
\begin{equation}\label{eq:Lipschitzian}
   |u(\mathbf{x})-u(\mathbf{x}_0)|\leq C' |\mathbf{x} -\mathbf{x} _0|, \qquad {\rm for \quad any}\quad  \mathbf{x} \in\overline{B^2} , \mathbf{x}_0\in\partial B^2 .
\end{equation}

Now for fixed  $\mathbf{y} \in {B^2}  $, let
 \begin{equation}\label{eq:H-y}
    H_{\mathbf{y} }(\mathbf{x} ):=\left\{
  \begin{array}{ll}
     \max\left\{ u(\mathbf{x}) ,u(\mathbf{x}+\mathbf{y} )-C' |\mathbf{y} |\right\}, \qquad & {\rm if}\quad \mathbf{x}\in  \overline{B^2} \cap( -\mathbf{y}+
     \overline{B^2}) ,\\
   u(\mathbf{x}),& {\rm if}\quad  \mathbf{x}\in  \overline{B^2} \setminus( -\mathbf{y}+ \overline{ B^2}) .
 \end{array}
  \right.
  \end{equation}
We claim that the restriction of  $H_{\mathbf{y} }$ to $  B^2$ belongs to $\mathscr {B }(B^2, \varphi)$. For a point $\mathbf{x}$ in the part of the boundary
$\partial( -\mathbf{y}+ \overline{B^2} )$
inside $B^2$, we have $\mathbf{x}\in   {B^2}$ and $ \mathbf{x}+\mathbf{y}\in  \partial{B^2}$. Thus $u(\mathbf{x}+\mathbf{y} )-C' |\mathbf{y} |\leq
u(\mathbf{x})$ by the boundary estimate
\eqref{eq:Lipschitzian}. Thus,  $ H_{\mathbf{y} }$ is $OPSH$ on $ {B^2}  $ by Proposition \ref{prop:QSH-m} (7).

On the other hand,  for $\mathbf{x}\in \partial{B^2} \cap ( -\mathbf{y}+ \overline{B^2} )$, we have   $\mathbf{x}\in \partial{B^2}$ and $ \mathbf{x}+\mathbf{y}\in
\overline{B^2}$, and so $u(\mathbf{x}) \geq u(\mathbf{x}+\mathbf{y} )-C' |\mathbf{y} | $ by the  boundary  estimate \eqref{eq:Lipschitzian} again. Therefore, $
H_{\mathbf{y} }(\mathbf{x} )=u(\mathbf{x})=\varphi(\mathbf{x})$ for $\mathbf{x}\in \partial{B^2}$. The claim is proved.
Consequently,  $H_{\mathbf{y} }(\mathbf{x} ) \leq u(\mathbf{x}) $ for all   $\mathbf{x}  \in  B^2$. In particular, for   $\mathbf{x} , \mathbf{x}+\mathbf{y}\in
B^2$
we have $u(\mathbf{x}+\mathbf{y} )-C' |\mathbf{y} |\leq u(\mathbf{x}) $, i.e.
\begin{equation*}
   u(\mathbf{x}+\mathbf{y} )- u(\mathbf{x}) \leq  C' |\mathbf{y} |.
\end{equation*}
By reversing the roles of $\mathbf{x} $ and $\mathbf{x}+\mathbf{y} $, we obtain $ u(\mathbf{x} )- u(\mathbf{x}+\mathbf{y}) \leq  C '|\mathbf{y} |$. Thus,
$u$ is a  Lipschitzian function with Lipschitzian constant $C' $.
\end{proof}

Now let us prove the interior estimate for second order difference of $u$.
\begin{prop}\label{prop:2diff}
    If $ \varphi \in C^2(  B^2)$ and $u = \Psi_{ B^2,\varphi}$, then for each
$\varepsilon \in (0,1)$,  there exists $ C > 0$
  such that for all $(\mathbf{x}, \mathbf{h}) \in B(\mathbf{0}, 1-\varepsilon) \times B(\mathbf{0}, \varepsilon/2) $
we have the estimate
\begin{equation}\label{eq:2diff}
   u(\mathbf{x} + \mathbf{h}) + u(\mathbf{x} - \mathbf{h}) - 2u(\mathbf{x}) < C|\mathbf{h}|^2 .
\end{equation}
\end{prop}
\begin{proof}
Fix $\varepsilon \in (0,1)$.
For $(\mathbf{a}, \mathbf{h}) \in B(\mathbf{0}, 1-\varepsilon) \times B(\mathbf{0}, \varepsilon/2) $,  define
\begin{equation*}
   L(\mathbf{a}, \mathbf{h},\mathbf{ x}): = T_{\mathbf{a}+\mathbf{h}}^{-1}\circ T_{\mathbf{a}}(\mathbf{x}),\qquad \mathbf{x}\in B^2,
\end{equation*}which is a diffeomorphism of  $ B^2$ to itself by   Corollary
\ref{cor:OPSH-Ta0}.
Then,   we see that
\begin{equation}\label{eq:U-a-h}\begin{split}
   U(\mathbf{a}  , \mathbf{h},\mathbf{x}):&=\left(T_{\mathbf{a}}^* \circ \left(T_{\mathbf{a}+\mathbf{h}}^{-1 }\right)^* u\right)(\mathbf{x}) = \left|\Psi_{\mathbf{a} }(\mathbf{x} \right)|^{-6}
  \left ((T_{\mathbf{a}+\mathbf{h}}^{-1 })^* u\right )\left(T_{\mathbf{a} } (\mathbf{x})\right) 
  = J_{\mathbf{a},\mathbf{h}}(\mathbf{x})u\left( L(\mathbf{a}, \mathbf{h},\mathbf{ x})\right)
 \end{split}  \end{equation}is $OPSH$ in $\mathbf{x}\in B^2$ by Proposition \ref{prop:OPSH-Ta-u}, where
\begin{equation*}
   J_{\mathbf{a},\mathbf{h}}(\mathbf{x}):= \left|\frac {\Psi_{\mathbf{a}+\mathbf{h}} ( L(\mathbf{a}, \mathbf{h},\mathbf{ x}))} { 
   \Psi_{\mathbf{a}}( \mathbf{x})  }\right|^6
\end{equation*}is continuous in $\mathbf{x}\in B^2$ by Proposition \ref{prop:Psi-a} (2)  and Corollary \ref{cor:OPSH-Ta} (2), when $(\mathbf{a}, \mathbf{h}) \in B(\mathbf{0}, 1-\varepsilon) \times B(\mathbf{0}, \varepsilon/2) $. Set
\begin{equation*}
  V(\mathbf{a}  , \mathbf{h},\mathbf{x}) :=\frac 12( U(\mathbf{a}  , \mathbf{h},\mathbf{x})+ U(\mathbf{a}  , -\mathbf{h},\mathbf{x})).
\end{equation*}which is also $OPSH$ on $ B^2$. In particular, for   $\mathbf{h} =0  $,
 \begin{equation*} L(\mathbf{a}, \mathbf{0},\mathbf{ x})=\mathbf{ x},\qquad
    V(\mathbf{a}  , \mathbf{0},\mathbf{x}) =u(\mathbf{x}).
 \end{equation*}

Note that $\mathbf{a} + \mathbf{h}\in \overline{ {B}}(\mathbf{0},1-\varepsilon/2)$ and $T_{\mathbf{a}+\mathbf{h}}^{-1}(\mathbf{x})\in \partial B^2$ if $  \mathbf{x}\in \partial B^2$.
For  $\mathbf{x}\in
\partial B^2$, by Proposition \ref{prop:Psi-a} (3) and Corollary \ref{cor:OPSH-Ta0} (3),
  we see that 
$V(\mathbf{a}  , \mathbf{h},\mathbf{x}) $ is $C^2(\overline{B}(\mathbf{0},1-\varepsilon) \times \overline{B}(\mathbf{0},\varepsilon/2) )$  with $C^2$ norms   only depending on  $\varepsilon$. Thus,  by Taylor's expansion for variable $\mathbf{h}$, we get
\begin{equation}\label{eq:V-V0}
  | V(\mathbf{a}  , \mathbf{h},\mathbf{x})- V(\mathbf{a}  , \mathbf{0},\mathbf{x})|\leq \frac {C_\varepsilon}2|\mathbf{h}|^2,
\end{equation}for   $\mathbf{a}   \in \overline{B}(\mathbf{0},1-\varepsilon) , \mathbf{h}\in  \overline{B}(\mathbf{0},\varepsilon/2)  $ and  $\mathbf{x}\in
\partial B^2$,
where
\begin{equation*}
  C_\varepsilon:=\sup_{} \left\{\left|\frac {\partial^2 V(\mathbf{a}  , \mathbf{h},\mathbf{x})}{\partial h_{\alpha p}\partial h_{\beta p'}}\right|; \mathbf{a}
  \in  \overline{B}(\mathbf{0},1-\varepsilon)  ,
   \mathbf{h}\in\overline{B}(\mathbf{0},\varepsilon/2),\mathbf{x}  \in\partial B^2,\alpha,\beta=1,2, p,p'=0,\dots,7\right\}<+\infty.
\end{equation*} Consequently, we have
  \begin{equation*}
     V(\mathbf{a}  , \mathbf{h},\mathbf{x})-   \frac {C_\varepsilon}2|\mathbf{h}|^2\leq V(\mathbf{a}  , \mathbf{0},\mathbf{x})=\varphi(\mathbf{x})\quad {\rm for}\quad  \mathbf{x}\in \partial B^2  .
 \end{equation*}
 
 On the other hand, $ V(\mathbf{a}  , \mathbf{h}, \cdot)$ is continuous on $\overline{B^2}$ for fixed  $\mathbf{a}, \mathbf{h}$ by by Proposition \ref{prop:Psi-a} (2) and Corollary \ref{cor:OPSH-Ta0} (2).
 Thus, for fixed $\mathbf{a}  , \mathbf{h}$,  we have  $V(\mathbf{a}  , \mathbf{h}, \cdot)-   \frac {C_\varepsilon}2|\mathbf{h}|^2\in\mathscr{B }(B^2 , \varphi)$. Therefore,
  \begin{equation}\label{eq:estimate-V}
     V(\mathbf{a}  , \mathbf{h},\mathbf{x})-   \frac {C_\varepsilon}2|\mathbf{h}|^2\leq u(\mathbf{x}) ,
 \end{equation}for all $(\mathbf{a}  , \mathbf{h},\mathbf{x})\in\overline{B}(\mathbf{0},1-\varepsilon) \times \overline{B}(\mathbf{0},\varepsilon/2)\times
 \overline{B^2 } $. Letting   $\mathbf{a}=\mathbf{x}\in{B}(\mathbf{0},1-\varepsilon) $ in \eqref{eq:U-a-h}, we get  
  \begin{equation}\label{eq:J}
    U(\mathbf{x}  , \mathbf{h},\mathbf{x})= J(  \mathbf{h},\mathbf{x})u(\mathbf{x} + \mathbf{h}) ,\qquad{\rm with}\quad J(  \mathbf{h},\mathbf{x}):=\left |\frac
    {\Psi_{\mathbf{x}+\mathbf{h}} ( \mathbf{x}+\mathbf{h} ) } {  \Psi_{\mathbf{x}}( \mathbf{x})  }\right|^6,
 \end{equation}by $T_{\mathbf{x} } (\mathbf{x})=\mathbf{0}$ and   $T_{\mathbf{x}+\mathbf{h}}^{-1}(\mathbf{0})=
 \mathbf{x}+\mathbf{h}$ by \eqref{eq:T--a}.
 Thus,  \eqref{eq:estimate-V} implies that for $(  \mathbf{h},\mathbf{x})\in  \overline{ {B}}(\mathbf{0},\varepsilon/2)\times \overline{ B}   (\mathbf{0},1-\varepsilon) $,
\begin{equation}\label{eq:J-h-x}\begin{split}
     J(  \mathbf{h},\mathbf{x}) u(\mathbf{x} + \mathbf{h}) +J(  -\mathbf{h},\mathbf{x}) u(\mathbf{x} - \mathbf{h}) - 2u(\mathbf{x}) \leq &   C_\varepsilon
     |\mathbf{h}|^2.
\end{split} \end{equation}

Note that by the expression \eqref{eq:J} of $J$, have  $J( \mathbf{0},\mathbf{x})=1$ and  $J(  \mathbf{h},\mathbf{x})$   smooth on
$\overline{B}(\mathbf{0}, \varepsilon/2) \times \overline{B}(\mathbf{0},1-\varepsilon)$ by Proposition \ref{prop:Psi-a} (1)-(2). So,
  \begin{equation*}
   | J(  \mathbf{h},\mathbf{x})- 1-\nabla J(  \mathbf{0},\mathbf{x})\cdot \mathbf{h}| \leq C'_\varepsilon|\mathbf{h}|^2
 \end{equation*} for some constant $C'_\varepsilon$ only depending on $\varepsilon$.
 Therefore, for $ (  \mathbf{h},\mathbf{x}) \in \overline{B}(\mathbf{0}, \varepsilon/2) \times \overline{B}(\mathbf{0},1-\varepsilon)$, we have
 \begin{equation*}\begin{split}
      u(\mathbf{x} + \mathbf{h}) + u(\mathbf{x} - \mathbf{h}) - 2u(\mathbf{x})   \leq & C_\varepsilon|\mathbf{h}|^2+|\nabla J(  \mathbf{0},\mathbf{x})\cdot
     \mathbf{h}|\cdot |u(\mathbf{x} + \mathbf{h}) -u(\mathbf{x} -
     \mathbf{h})|\\
     &+C'_\varepsilon|\mathbf{h}|^2 (| u(\mathbf{x} + \mathbf{h})| + |u(\mathbf{x} - \mathbf{h})|) = O(|\mathbf{h}|^2),
\end{split} \end{equation*}
by \eqref{eq:J-h-x} and
  $u\in Lip( \overline{ B^2})$  by Proposition \ref{prop:Lip}.
 The estimate \eqref{eq:2diff} follows.
\end{proof}

For $u\in L^\infty_{loc}(\Omega)$,  where $\Omega$ is a domain of $\mathbb{R}^N$,  define a smoothing of $u$
\begin{equation*}
  u_\varepsilon(\mathbf{x}):=\frac 1{\tau_N \varepsilon^N}\int_{|\mathbf{y}|<\varepsilon  }u(\mathbf{x}+\mathbf{y})dV(\mathbf{y})   \qquad{\rm and} \qquad
  T_{\varepsilon }u: =2(N+2)\frac {u_\varepsilon -u}{\varepsilon^2},
\end{equation*}where $\tau_N $ is the volume of unit ball.

\begin{thm}\label{thm:C11}
  Suppose that $ \varphi \in C^2( B^2)$ and $u = \Psi_{B^2,\varphi}$. Then the
weak second order derivatives  of $u$ belong to $ L^\infty_{loc}$, i.e.  $u\in C_{loc}^{1,1} (B^2)$.
\end{thm}
\begin{proof} Since $u$ is $OPSH$,
\begin{equation}\label{eq:T-varepsilon} \begin{split}
0\leq T_{ \varepsilon }u = &\frac {36}{\tau_{16} \varepsilon ^{18}}\int_{B(\mathbf{0},\varepsilon )}\left(u(\mathbf{x}+   \mathbf{t})-
u(\mathbf{x})\right)dV(\mathbf{t}) \\
 = &\frac {36}{\tau_{16} \varepsilon ^{18}}\int_{B(\mathbf{0},\varepsilon )}\frac 12 \left(u(\mathbf{x}+ \mathbf{t})+u(\mathbf{x}-  \mathbf{t})- 2
 u(\mathbf{x})\right)dV(\mathbf{t}) \\
 \leq & \frac {18}{ \varepsilon ^{18}}\int_0^\varepsilon  Cr^2 r^{15}dr\leq C,
 \end{split}\end{equation}
  by    Proposition \ref{prop:2diff}. On the other hand, it is direct to see that
  \begin{equation*}\begin{split}
 \int_{\mathbb{O}^2} T_{ \varepsilon }u\cdot \varphi dV=\int_{\mathbb{O}^2}u\cdot T_{ \varepsilon }\varphi dV\rightarrow\int_{\mathbb{O}^2}u\cdot\Delta\varphi dV,
 \end{split}\end{equation*}
 for $\varphi\in C_c^2(B^2)$. This together with estimate \eqref{eq:T-varepsilon} implies $\triangle u\in L^\infty_{loc}(B^2)$. 
 Recall that for a domain $\Omega$  of $\mathbb{R}^N$, if $u$ is subharmonic and  $\triangle u\in L^\infty_{loc}(\Omega) $, then $\lim_{\varepsilon\rightarrow
0}  T_{\varepsilon }u =\triangle u $ in $L^\infty_{loc}(\Omega) $ \cite[Corollary 4.2.3]{klimek}.

  By using Riesz transformation, we see that  $\frac {\partial^2 u}{\partial x_{\alpha
k}\partial x_{\beta l}}  \in   L^p_{loc} (B^2)$ for any $1<p<\infty$. But if $\frac {\partial^2 u}{ \partial
  x_{1k}^2}\in  L^p_{loc} (B^2)$, we always have
\begin{equation*}
  \frac {u(\mathbf{x}_1+\delta e_k, \mathbf{x}_2)+u(\mathbf{x}_1-\delta e_k, \mathbf{x}_2) -2 u(\mathbf{x})}{ \delta^2 }\rightarrow\frac {\partial^2 u}{ \partial
  x_{1k}^2} (\mathbf{x})
\end{equation*}
in  $ L^p_{loc} (B^2)$ as $\delta\rightarrow 0$. Therefore, there exists a subsequence converges almost everywhere in $ B(\mathbf{0},
1-\varepsilon)$ for fixed $\varepsilon>0$. Thus, by the estimate
\eqref{eq:2diff},
 $\frac {\partial^2 u}{ \partial x_{\alpha k}^2} \leq C  $ almost everywhere on $ B(\mathbf{0},
1-\varepsilon)$ for some constant $C>0$. Consequently, $0\leq \triangle_\alpha u\leq 8C$. This also implies $\frac {\partial^2 u}{ \partial x_{\alpha k}^2}   $ are also lower bounded, because
\begin{equation*}
  - \frac {\partial^2 u}{ \partial x_{\alpha p}^2}\leq \sum_{k\neq p}\frac {\partial^2 u}{ \partial x_{\alpha k}^2}\leq 7 C.
\end{equation*}
To see the mixed   partial derivatives $\frac {\partial^2 u}{\partial x_{\alpha k}\partial x_{\beta l}}$ locally  bounded, let $X=\frac {\partial }{ \partial
x_{\alpha k} }+ \frac {\partial }{ \partial x_{\beta l} } $. Then $X^2u$ is
locally  bounded  by $C$ by the above argument. But
\begin{equation*}
  2 \frac {\partial^2 u}{\partial x_{\alpha k}\partial x_{\beta l}}=X^2u-\frac {\partial^2 u}{ \partial x_{\alpha k}^2} -\frac {\partial^2 u}{ \partial x_{\beta
  l}^2}.
\end{equation*}So $\frac {\partial^2 u}{\partial x_{\alpha k}\partial x_{\beta l}}$ is locally  bounded.
\end{proof}

The following result  due to Bedford-Taylor \cite[Theorem 7.3]{BT}
  plays a crucial role in the   proof of the
Perron-Bremermann function satisfying the homogeneous Monge-Amp\`ere
equation.

\begin{prop} \label{prop:2-derivatives} Suppose that   $\Omega$ is a  bounded   open subset of  $\mathbb{R}^N$.  Let $u: \Omega\rightarrow \mathbb{R}$ be a
subharmonic function such that
$\frac {\partial^2 u}{\partial x_a \partial x_b}\in L^\infty_{loc}(\Omega)$ for $a,b = 1, \dots ,N$.
  Let $\eta > 0$  and let $\{\varepsilon_j\}$ be a sequence of positive numbers converging
to zero. Then there exists a compact set $ K\subset \Omega$ and a natural number $j_0 $
such that;
\\
(i) $|\Omega\setminus K|<\eta$;
\\
(ii)
$\left.\frac {\partial^2 u}{\partial x_a \partial x_b}\right|_K\in C(K)$ for $a,b = 1, \dots ,N$;
\\
(iii) for all $ x \in K$,  $j> j_0 $ and $y \in \overline{B}( {x},\varepsilon_j)$,
\begin{equation*}
  \left |T_{\varepsilon_j}u( {y})-\triangle u( {x})\right|\leq \eta,\qquad \left|\frac {\partial^2 u_{\varepsilon_j}}{\partial x_a \partial x_b}(y)-\frac
  {\partial^2 u }{\partial x_a \partial
  x_b}(x)\right|\leq \eta,\qquad a,b = 1, \dots ,N.
\end{equation*}
\end{prop}

\begin{proof}[Proof of Theorem \ref{thm:Dirichlet-ball}]
Since $u\in L^\infty_{2,loc}(B^2)$ by Theorem \ref{thm:C11}, we can write $ \det(Hess_{\mathbb{O}}(u)) =gdV$ for some non-negative function $g\in L^\infty_{ loc}(B^2)$.
 Suppose that $g$ does not vanish. Then there exists some constant $c\in (0,1)$ such that
  \begin{equation*}
    |\{\mathbf{x}\in B^2; g(\mathbf{x})>c \}|>20\tau_{16}c.
 \end{equation*}

Let $\{\varepsilon_j\}$ be a sequence of positive numbers converging
to zero and let
\begin{equation}\label{eq:esssup}
  M:={\rm esssup}\left\{|\overline{\partial}_\alpha\partial_\beta u (\mathbf{y})|; |\mathbf{y}|\leq1-\frac c2,\alpha,\beta=1,2\right\},\qquad 
   a<\frac c{2M+1},\qquad \eta<\frac a5 .
\end{equation}
 Then, by Proposition \ref{prop:2-derivatives},   there exists a compact set $ K\subset B^2$ and a natural number $j_0 $
such that $(i)$
 $|B^2\setminus K| <\eta$, $(ii)$ $\overline{\partial}_\alpha\partial_\beta u\in C(K)$,
 and $(iii)$ for all $ \mathbf{x} \in K$,  $j> j_0 $ and $\mathbf{y} \in \overline{B}(\mathbf{x},\varepsilon_j)$,
\begin{equation}\label{eq:2-o-derivatives} \begin{split}
  \left |T_{\varepsilon_j}u(\mathbf{y})-\triangle u(\mathbf{x})\right|\leq \eta&,
 \\
\overline{\partial}_\alpha\partial_\beta u_{\varepsilon_j} (\mathbf{y})=\overline{\partial}_\alpha\partial_\beta  u  (\mathbf{x})+\mathcal{E}_{\alpha\beta}(\mathbf{x}, \mathbf{y})&,\qquad{\rm with }\quad
  \left|\mathcal{E}_{\alpha\beta}(\mathbf{x}, \mathbf{y})\right|\leq \eta,\qquad \alpha,\beta=1,2.
\end{split} \end{equation}Here $\mathcal{E}_{\alpha\beta}$ is defined by the second identity. Since $|B(\mathbf{0}, 1  )\setminus B(\mathbf{0}, 1 - c)|<16\tau_{16}c$
The choice of $\eta$ in \eqref{eq:esssup} implies  that we have  $ | K \cap B(\mathbf{0}, 1 - c)\cap\{\mathbf{x}\in B^2; g(\mathbf{x})>c \}|\neq 0$ and  $g\in C(K)$. So there exists a point $\mathbf{x}_0 \in K \cap
B(\mathbf{0}, 1 - c)$ such
that $u$ is second-order differentiable at this point, $g(\mathbf{x}_0 ) > c$ and $|\overline{\partial}_\alpha\partial_\beta u (\mathbf{x}_0)|<M$. Now
 let
\begin{equation*}
 v_{ j}(\mathbf{x})=   u_{\varepsilon_j}(\mathbf{x})-(a-2\eta)|\mathbf{x}-\mathbf{x}_0|^2/8+(a-4\eta)\varepsilon_j^2/8-\varepsilon_j^2\frac {\triangle
 u(\mathbf{x}_0)}{36},
\end{equation*}for $\mathbf{x}\in {B}(\mathbf{x}_0, \varepsilon_j)$.
Then for $\mathbf{x}\in {B}(\mathbf{x}_0, \varepsilon_j)$,
\begin{equation}\label{eq:Hessian-x-x0}
 \left(\overline{\partial}_\alpha\partial_\beta  v_{ j}(\mathbf{x})\right)= \left(\overline{\partial}_\alpha\partial_\beta u(\mathbf{x}_0)\right)-aI+
 (\mathcal{E}_{\alpha\beta})+2\eta I
\end{equation}by \eqref{eq:2-o-derivatives}.
Let $\lambda_2\geq\lambda_1 $ be two eigenvalues of $\left(\overline{\partial}_\alpha\partial_\beta u(\mathbf{x}_0)\right)$, i.e. they are solutions to
\begin{equation*}
   \lambda^2- {\rm tr} \left(\overline{\partial}_\alpha\partial_\beta u(\mathbf{x}_0)\right)\lambda+ \det \left(\overline{\partial}_\alpha\partial_\beta u(\mathbf{x}_0)\right)=0.
\end{equation*}
Since   $\left(\overline{\partial}_\alpha\partial_\beta u(\mathbf{x}_0)\right)$ is nonnegative,  we must have  $\lambda_2\geq\lambda_1>0$ by Proposition  \ref{prop:non-negative}, and 
\begin{equation*}
    \lambda_1< \lambda_1+\lambda_2 =tr \left(\overline{\partial}_\alpha\partial_\beta u(\mathbf{x}_0)\right)\leq 2M, 
\end{equation*}by \eqref{eq:esssup}. Then,  $\det\left(\overline{\partial}_\alpha\partial_\beta u(\mathbf{x}_0)\right)= g(\mathbf{x}_0 )>c>0 $ implies
\begin{equation*}
    2M \lambda_1\geq \lambda_1 \lambda_2 =g(\mathbf{x}_0 ) > c.
\end{equation*}Thus, $\lambda_1\geq  \frac c{2M}>a$, and so  $ \left(\overline{\partial}_\alpha\partial_\beta u(\mathbf{x}_0)\right)-aI$ is also positive.
Consequently, \eqref{eq:Hessian-x-x0} is  positive and so
$v_j\in  OPSH( {B}(\mathbf{x}_0, \varepsilon_j)) $. But we can write
\begin{equation}\label{eq:vj-u}
 v_{ j}(\mathbf{x})- u (\mathbf{x})= \frac 1{36} (T_{\varepsilon_j}u(\mathbf{x}) -\triangle u(\mathbf{x}_0))
 \varepsilon_j^2-(a-2\eta)|\mathbf{x}-\mathbf{x}_0|^2/8+(a-4\eta)\varepsilon_j^2 /8.
\end{equation}
So if  $\mathbf{x}\in \partial{B}(\mathbf{x}_0, \varepsilon_j)$, we have
\begin{equation*}
 v_{ j}(\mathbf{x})- u (\mathbf{x}) \leq \varepsilon_j^2\left(\frac 1{36} \eta - \frac 1{4}\eta\right)<0,
\end{equation*} Therefore $v_{ j}\leq u$ on $  {B}(\mathbf{x}_0, \varepsilon_j)  $ by the maximality of $u$  on $B^2$. But \eqref{eq:vj-u} together with
\eqref{eq:2-o-derivatives} implies
\begin{equation*}
 v_{ j}(\mathbf{x}_0)- u (\mathbf{x}_0)\geq\frac  { \varepsilon_j^2}8\left(a-4\eta-\frac 2{9} \eta  \right)>0,
\end{equation*}by \eqref{eq:eta}.
This is a contradiction.
\end{proof}

  Theorem \ref{thm:Dirichlet-ball} can be generalized  to  general strictly octonionic pseudoconvex  domains  as in \cite[Theorem 8.2-8.3]{BT}
for complex Monge-Ampere equation. We omit details.

\begin{thm} \label{thm:Dirichlet-ball-1}  Suppose that $\Omega$ is a strictly  octonionic   pseudoconvex  domain in $\mathbb{O}^2$. For  $\varphi \in C( \Omega)$,
the  Perron-Bremermann  function $ \Psi_{\Omega,\varphi}
$ is  $C (\overline{\Omega})$ and
 is the unique solution to    the   Dirichlet
problem \begin{equation}\label{eq:Dirichlet-problem2}
  \left\{
  \begin{array}{ll} u\in  OPSH(\Omega) \cap  C( \overline{ \Omega}),\qquad&\\
     \det(Hess_{\mathbb{O}}(u)) = 0,\qquad & {\rm on}\quad \Omega,\\
    u= \varphi,& {\rm on}\quad \partial \Omega.
 \end{array}
  \right.
\end{equation}
\end{thm}
\begin{rem} The weighted transformation formula of quaternionic Monge-Amp\`ere  operator can be found in 
   \cite{wang25}. So the  $C_{loc}^{1,1}$-regularity of solution to the homogeneous quaternionic
Monge-Amp\`ere equation on the unit ball also holds.
\end{rem}

 \section{Octonionic relative extremal function  and   octonionic  capacity}
\subsection{Octonionic relative extremal function  } For a domain $\Omega$ in $\mathbb{O}^2$ and $E\subset \Omega$, let
 \begin{equation}\label{eq:U}
   \mathcal{U}( E ,\Omega): =\{  u\in  OPSH(\Omega) , u\vert_{\Omega}\leq    0,u\vert_{E}\leq   -1 \},
 \end{equation}
and
$
    \omega (\mathbf{x},E,\Omega):=\sup \{u(\mathbf{x}); u\in \mathcal{U}( E ,\Omega) \},
$
whose   upper semicontinuous regularization $ \omega^* (\mathbf{x},E,\Omega)$  is called the {\it (octonionic) relative extremal function}  of the set $E$ in
$\Omega$. It is $OPSH $ by Proposition
\ref{prop:QSH-m} (6).  

A domain $\Omega$ is called {\it  octonionic    hyperconvex}  if there exists
a continuous function  $\varrho\in   OPSH (\Omega)$  such that $\varrho < 0 $ in $\Omega$, $\lim_{\mathbf{x}\rightarrow \partial\Omega} \varrho(\mathbf{x})=0$ and
$\{\mathbf{x}\in\Omega;\varrho(\mathbf{x})<c\}$ is
relatively
compact in $ \Omega$ for any $c < 0$. It is called  {\it strongly octonionic  hyperconvex}
if there exists $\varrho\in      OPSH (G)$ satisfying the above condition for some open set $  G \Supset\Omega$.
The relative extremal function has the following simple properties.

  \begin{prop} \label{prop:relative-extremal}
(1) {\it (Monotonicity)  if $E_1\subseteq E_2\subset\Omega$, then  $ \omega^* (\mathbf{x},E_1,\Omega)  \geq \omega^* (\mathbf{x},E_2,\Omega) $;  if $E \subseteq
\Omega_1\subset  \Omega_2$, then  $ \omega^*
(\mathbf{x},E ,\Omega_1)  \geq \omega^* (\mathbf{x},E ,\Omega_2) $ for $\mathbf{x}\in \Omega_1$}.

(2)  {\it $ \omega^* (\mathbf{x},E,\Omega)\equiv 0$  if and only if $E$ is octonionic  polar in $\Omega$}, i.e. $E$  is contained in
$\{\mathbf{x};v(\mathbf{x})=-\infty\}$ for some $OPSH $ function $v$.

(3) {\it  Let $    \Omega =\{\varrho<0\}$ be    octonionic  hyperconvex.
If   $ E\Subset  \Omega $, then
$ \omega^* (\mathbf{x},E,\Omega)\rightarrow 0$ as $\mathbf{x}\rightarrow \partial\Omega$}.

(4)  {\it Let $    \Omega =\{\varrho<0\}$ be   strongly octonionic  hyperconvex.
If   $ E\Subset  \Omega $,
then  $ \omega^* (\cdot,E,\Omega)$ admits an  $OPHS$ extension to a neighborhood of   $\overline{\Omega}$}.
\end{prop}
\begin{proof} (2) The proof is the same as  the complex case \cite{klimek}.

(3) Note that $M\varrho \in\mathcal{U}( E ,\Omega)$ for a suitable $M>0$ since $E\Subset \Omega$. Then
\begin{equation}\label{eq:omega*varrho}
   0\geq \omega^* (\cdot,E,\Omega)\geq M\varrho,
\end{equation}
on $\Omega$, and so we must have
 $\omega^* (\mathbf{x},E,\Omega)\rightarrow 0$ as $\mathbf{x}\rightarrow \partial\Omega$.

(4) By  \eqref{eq:omega*varrho} above, the $  OPSH $  function
\begin{equation*}
  w( \mathbf{x})=\left\{ \begin{array}{ll}
  \omega^* (\mathbf{x},E,\Omega) ,\qquad &\mathbf{x}\in \Omega,
  \\ M\varrho,\qquad &\mathbf{x} \in \Omega^c,
   \end{array}\right.
\end{equation*}
gives an extension to a neighborhood of $\overline{\Omega}$.
\end{proof}

A point $\mathbf{x}_0\in   K$ is called an {\it (octonionic)  regular point} of a compact  subset $K \Subset\Omega$ if  $\omega^* (\mathbf{x}_0,K,\Omega) = -1$. A
compact subset $K\Subset
\Omega$ is called {\it  (octonionic)  regular} in $\Omega$  if each point   of $   K$ is  octonionic  regular.

\begin{prop} \label{prop:regular} Suppose that    $\Omega$  is an octonionic  hyperconvex  domain.  Let $K$ be a compact subset of  $\Omega$  and be the union of a
family of closed balls. Then
$\omega^*  (\cdot,K,\Omega) $ is continuous and $\omega^*  (\cdot,K,\Omega)|_K\equiv- 1$.
\end{prop}
\begin{proof} See \cite{klimek} for the proof in the complex case. We claim that $\omega
(\mathbf{x},\overline{B}(\mathbf{a},r),{B}(\mathbf{a},R)) $, the relative extremal function for concentric balls $\overline{{B}}(\mathbf{a},r) \subset {B}(\mathbf{a},R)$,  is
given by
\begin{equation}\label{eq:u-r-R}
u_0(\mathbf{x}):=\max \left\{ \frac 1{  \frac{1}{ r^{6}}-\frac 1{ R ^6} } \left(-\frac{1}{ | \mathbf{x}-\mathbf{a}  |^{6}}+\frac 1{
R ^6}\right)
, -1\right\}.
\end{equation}
Since $-\frac{1}{ |\mathbf{x}-\mathbf{a} |^{6}}
$
  is $OPSH$ on $\mathbb{O}^2$ by Proposition \ref{prop:fundamental-solution},  so is $u_0$. The set of  octonionic (right) lines in $
  \mathbb{O}^2 $
  can be parameterized by $\mathbf{b}=(b_1,\mathbf{b}_2)$ with $b_1\in \mathbb{R},\mathbf{b}_2\in \mathbb{O}$ and $b_1^2+|\mathbf{b}_2|^2=1$:
  \begin{equation}\label{eq:octonionic-line}
 \{\mathbf{a}+\mathbf{b}\mathbf{t};\mathbf{t}\in \mathbb{O} \} ,
\end{equation}
Note that for any $v\in  OPSH({B}(\mathbf{a},R))$ such that $ v\vert_{{B}(\mathbf{a},R)}\leq    0$ and $v\vert_{{B}(\mathbf{a},r)}\leq   -1 $ and for given $\mathbf{b}
$ as above,   the function $v$ is subharmonic on the
octonionic   line \eqref{eq:octonionic-line}. So
$
  f(\mathbf{t})=v(\mathbf{a}+\mathbf{b}\mathbf{t})-u_0(\mathbf{a}+\mathbf{b}\mathbf{t})$
is subharmonic on the annulus $D(\mathbf{0},R)\setminus \overline{D}(\mathbf{0},r)$ in $\mathbb{O}$. This is  because
\begin{equation*}
   u_0(\mathbf{a}+\mathbf{b}\mathbf{t})=\frac{ -\frac{1}{ |
\mathbf{t}  |^{6}}+\frac 1{
R ^6} } {  \frac{1}{ r^{6}}-\frac 1{ R ^6} }
\end{equation*}
  is harmonic in $\mathbb{O}\setminus \overline{D}(\mathbf{0},r)$. Moreover, by the assumption of $v$,  $\limsup_{\mathbf{t}\rightarrow \mathbf{c}} f(\mathbf{t})\leq 0$ for each $\mathbf{c}$ in
  the boundary of $D(\mathbf{0},R)\setminus
\overline{D}(\mathbf{0},r)$. It follows from the maximum principle for subharmonic functions that $ f \leq 0$ on $D(\mathbf{0},R)\setminus
\overline{D}(\mathbf{0},r)$. Since each
point of ${B}(\mathbf{a},R)\setminus
{B}(\mathbf{a},r)$ belongs to such an  annulus in some octonionic   line  \eqref{eq:octonionic-line}, we see that  $ v(\mathbf{x}) \leq u_0(\mathbf{x})$ on
${B}(\mathbf{a},R) \setminus
{B}(\mathbf{a},r) $. Therefore, $u_0$ is exactly the relative extremal function $\omega  (\cdot,\overline{{B}}(\mathbf{a},r) , {B}(\mathbf{a},R))$. It is continuous.

Let $\mathbf{y}\in K$. By assumption, there exists $\mathbf{a}\in K$ and $R>r>0$ such that $\mathbf{y}\in\overline{B}(\mathbf{a},r)\subset K$ and
${B}(\mathbf{a},R)\subset\Omega$. Now if $\mathbf{x}\in
{B}(\mathbf{a},R)$, we have
\begin{equation*}
 \omega  (\mathbf{x},K,\Omega) \leq   \omega  (\mathbf{x},\overline{B}(\mathbf{a},r),\Omega) \leq   \omega
 (\mathbf{x},\overline{B}(\mathbf{a},r),{B}(\mathbf{a},R)) ,
\end{equation*}by the monotonicity in Proposition \ref{prop:regular}.  But,
 \begin{equation*}
    \omega  (\cdot,\overline{{B}}(\mathbf{a},r) , {B}(\mathbf{a},R))|_{\partial{{B}}(\mathbf{a},r)}=-1\qquad {\rm   and }\qquad  \omega
(\cdot,\overline{{B}}(\mathbf{a},r) , {B}(\mathbf{a},R))|_{\partial{{B}}(\mathbf{a},R)}=0,
 \end{equation*}
   by \eqref{eq:u-r-R}. Thus
$\limsup_{\mathbf{x}\rightarrow\mathbf{y}} \omega ^* (\mathbf{x},K,\Omega)\leq-1$. It is obvious that $\omega ^* (\mathbf{y},K,\Omega)\geq-1$.

To show the continuity, note that we can choose a continuous defining function $\varrho$ of  the hyperconvex  domain such that $\varrho\leq u:=\omega^*
(\cdot,K,\Omega)$ in $\Omega$. Given $\varepsilon>0$,
there  exists $\eta>0$ such that $u- \varepsilon<\varrho$ on $\Omega\setminus\Omega_\eta$ and $K\subset \Omega_\eta:=\{\mathbf{x}\in \Omega;{\rm
dist}(\mathbf{x},\partial \Omega)>\eta\}$.
By the smooth approximation of $OPSH$, $u _\delta\downarrow u$ as $\delta\downarrow  0$, and Dini's theorem, we see that for sufficiently small $\delta>0$, we
have $u _\delta-2\varepsilon<\varrho$ on
$\partial\Omega_\eta$ and
$u _\delta-2\varepsilon<-1$ on $K$. Now define
\begin{equation*}
 v_\varepsilon( \mathbf{x})=\left\{ \begin{array}{ll}
 \varrho ,\qquad & {\rm in}\quad  \Omega\setminus\Omega_\eta,
  \\\max\{u _\delta-2\varepsilon,\varrho\},\qquad &  {\rm in}\quad  \Omega_\eta,
   \end{array}\right.
\end{equation*}which is also continuous. Then $v_\varepsilon \in\mathcal{U}( K ,\Omega)$ and $u -2\varepsilon \leq  v_\varepsilon\leq u$. Since
$\varepsilon>0$ is arbitrarily chosen, $u$ is continuous.
\end{proof}

\begin{cor} \label{cor:regular} For any    compact  subset $ K$ of an open set  $U$,
there exists an  octonionic  regular compact subset $ E$ such that $K \subset E \Subset U$. In particular, $\omega^*  (\mathbf{x},E,\Omega)=-1$ for $\mathbf{x} \in
E$.
\end{cor}\begin{proof} Take $E$ to be
 $K_\eta:=\{\mathbf{x}\in \Omega;{\rm dist}(\mathbf{x},K)\leq\eta\}=\bigcup_{\mathbf{a}\in K}\overline{B}(\mathbf{a},\eta)$ for sufficiently small $\eta>0$ and use
 Proposition
 \ref{prop:regular}.
\end{proof}

\begin{prop} \label{prop:maximal}  Let $K$ be an octonionic  regular compact subset of   an octonionic  hyperconvex  domain $\Omega$.
  Then, (1) $\omega^*( \cdot, K, \Omega)\in C(\Omega)$; (2) relative   extremal function $\omega^*( \cdot, K, \Omega)   $ is maximal in $ \Omega\setminus K$;
  (3)
  \begin{equation}\label{eq:omega=0}
   \det(Hess_{\mathbb{O}}( \omega^*( \cdot,  K, \Omega)  )  )=0\qquad \text { on }\quad \Omega\setminus K.
\end{equation}
\end{prop}
\begin{proof} (1) Consider $\Omega_j:=\{\mathbf{x}\in\Omega; \omega^*(\mathbf{x}, K, \Omega)<-1/j\}$ for positive integers $j$. Then $\Omega_j\subset \Omega_{j+1}$
and $\Omega_j\Subset\Omega$
by \eqref{eq:omega*varrho} for hyperconvex  domain $\Omega$. Fixed a $j_0$,  the relative   extremal function can be approximated on  $\overline{\Omega}_{j_0}$ by smooth $OPSH   $
  functions $v_{t}\downarrow  \omega^* (\cdot ,K,\Omega)$. Applying Hartogs' lemma \cite[Theorem 2.6.4]{klimek} for subharmonic functions twice to this sequence, we  see that
  there exists $t_0$ such that for
  $t> t_0$, we have $v_t\leq0$ on  $\overline{\Omega}_{j_0}$ and simultaneously, $v_t\leq -1+ 1/j_0$ on  $K$. Then the function
  \begin{equation*}
   \widetilde{w}(\mathbf{x})=\left\{\begin{array}{ll}
   \max \left\{v_t(\mathbf{x} )-1/j_0,\omega^*(\mathbf{x}, K, \Omega)\right\},\qquad &\mbox{if }\mathbf{x}\in {\Omega}_{j_0},\\ \omega^*(\mathbf{x}, K, \Omega)
   ,\qquad &\mbox{if }
   \mathbf{x}\notin {\Omega}_{j_0},
\end{array}
\right.
\end{equation*}is $    OPSH $ by Proposition \ref{prop:QSH-m}, and so belongs to
$\mathcal{U}( K ,\Omega)$. Thus, 
\begin{equation*}
   \omega^*(\mathbf{x}, K, \Omega) -1/j_0\leq v_t(\mathbf{x} )-1/j_0\leq  \widetilde{w}(\mathbf{x})\leq \omega^*(\mathbf{x}, K, \Omega)\qquad {\rm for}\quad \mathbf{x}\in \overline{\Omega}_{j_0} .
\end{equation*}
Consequently, $v_t $ converges uniformly to $  \omega^*(\cdot, K, \Omega)$ on compact
subsets of $\Omega$. So it is continuous.

(2) Suppose that $\omega^*(\cdot, K, \Omega)  $ is not maximal. Then
  there exists a domain $G \Subset  D\setminus K$ and a
function $v\in   OPSH(G)$ such that $\underline{\lim}_{\mathbf{x}\in \partial G}(\omega^*(\mathbf{x}, K, \Omega) - v(\mathbf{x})) \geq 0$, but $v(\mathbf{x}_0) >
\omega^*(\mathbf{x}_0, K, \Omega)$
at some point $\mathbf{x}_0\in G$. Since $\omega^*(\cdot, K, \Omega)|_K\equiv-1 $ by (1), the function
\begin{equation*}
   w(\mathbf{x})=\left\{\begin{array}{ll}
   \max \left\{ v(\mathbf{x} ),\omega^*(\mathbf{x}, K, \Omega)\right\},\qquad &\mbox{if } \mathbf{x}\in G,\\\omega^*(\mathbf{x}, K, \Omega) ,\qquad &\mbox{if }
   \mathbf{x}\notin G,
\end{array}
\right.
\end{equation*}   belongs to
   $  \mathcal{U}( K ,\Omega)$ by definition, and so $w\leq \omega^*(\cdot, K, \Omega)$. This contradicts to $w(\mathbf{x}_0)=v(\mathbf{x}_0) >
\omega^*(\mathbf{x}_0, K, \Omega)$.

 (3) Suppose that $  \det(Hess_{\mathbb{O}}(\omega^*(\cdot,  K, \Omega)  ))$ does not vanish on $\Omega\setminus K$. There exists a ball $B(\mathbf{x}_0,r)\subset
 \Omega\setminus K$
 where
  \begin{equation}\label{eq:not-equiv-0}
   \det(Hess_{\mathbb{O}}( \omega^*(\cdot ,  K, \Omega)  )) \neq 0
 \end{equation}as a measure.
 Let $v(\mathbf{x})
$ be the continuous  Perron-Bremermann  solution to the   Dirichlet
problem  with continuous boundary value:
\begin{equation*}
  \left\{
  \begin{array}{ll}
     \det(Hess_{\mathbb{O}}(v)) = 0,\qquad & {\rm on}\quad B(\mathbf{x}_0,r),\\
    v= \omega^*(\cdot,  K, \Omega),& {\rm on}\quad \partial B(\mathbf{x}_0,r),
 \end{array}
  \right.
\end{equation*}given by Theorem \ref{thm:Dirichlet-ball-1}.  It is
is maximal by construction. In particular,  $v \geq \omega^*( \cdot,  K, \Omega)$ on $B(\mathbf{x}_0,r)$.   But $v\not\equiv\omega^*(\cdot,  K, \Omega)$ on
$B(\mathbf{x}_0,r)$ by  \eqref{eq:not-equiv-0}.
 Therefore, $v(\mathbf{x}') >  \omega^*( \mathbf{x}' ,  K, \Omega)$ for some $ \mathbf{x}' \in B(\mathbf{x}_0,r)$. But
 \begin{equation*}
    w(\mathbf{x})=\left\{ \begin{array}{ll}
  \omega^* (\mathbf{x},E,\Omega) ,\qquad &\mathbf{x}\in \Omega\setminus B(\mathbf{x}_0,r),
  \\ \max\{v(\mathbf{x} ),\omega^* (\mathbf{x},E,\Omega)\},\qquad &\mathbf{x}\in B(\mathbf{x}_0,r),
   \end{array}\right.
 \end{equation*}belongs to
 $ \mathcal{U}( K ,\Omega) $. Then $w(\mathbf{x}') >  \omega^*( \mathbf{x}' ,  K, \Omega)$ contradicts to the
  maximality of $\omega^*( \cdot,  K, \Omega)$ in (2).
\end{proof}

  \subsection{ Octonionic  capacity} See \cite[Section 3]{sadullaev1} for complex   capacity. The
{\it octonionic exterior capacity } of a set $E\subset\Omega$ is defined as
$
   C  ^*(E) = \inf\{ C  (U);  \mbox{ open   } U\supset E\}
$.
 It is obviously monotonic  by definition.
\begin{prop}\label{prop:capacity}
Let $\Omega$ be an   octonionic  hyperconvex domain in $\mathbb{O}^2$. Then,

(1)
For any octonionic  regular compact subset $K \subset \Omega$,  \eqref{eq:capacity-K} holds.

(2)  For any compact subset  $K \subset \Omega$, $C  (K)=\inf\{C  (E); \Omega \supset E\supset K$ and $ E $ is an octonionic  regular   compact subset$\}$. In
particular, $C  ^*(K)=C
(K)$.

 (3) If $K$ is an octonionic  regular compact subset, then
\begin{equation}\label{eq:capacity-K-2}
   C  (K)= \sup_{u_1,u_2\in  \mathscr  L}  \int_{K}   \det(Hess_{\mathbb{O}}(u_1 ),Hess_{ \mathbb{O}}(u_2 )) ,
\end{equation}where
$
 \mathscr  L:=\left\{  u \in OPSH(\Omega)\cap C(\Omega); -1\leq u < 0\right\}.
$

(4) Suppose that $\Omega$ is strongly octonionic  hyperconvex.   If $U\subset \Omega$ is an open set, then
\begin{equation}\label{capacity defi}
 C  (U) =\sup_{u_1,u_2 \in  \mathscr  L} \int_{U}  \det(Hess_{\mathbb{O}}(u_1), Hess_{\mathbb{O}}(u_2))  =\sup_{u_1,u_2 \in  \mathscr  L_\infty} \int_{U}
 \det(Hess_{\mathbb{O}}(u_1),Hess_{\mathbb{O}}(u_2)),
\end{equation}
where
$
 \mathscr  L_\infty:=\left\{  u\in OPSH(\Omega)\cap C^{\infty}(\Omega); -1\leq u < 0\right\}.
$

(5) The exterior capacity   is monotonic, i.e.  if $E_1\subseteq E_2$, then  $C  ^*(E_1)\subseteq C  ^*(E_2)$, and
countably subadditive, i.e. $C  ^*(\cup_j E_j)\leq \sum_j C  ^*(  E_j)$.

(6) If $U_1\subset U_2\subset\dots$ are open subsets of $\Omega$, then
$
C  \left(\bigcup_{j=1}^\infty U_j,\Omega\right)=\lim\limits_{j\to\infty} C  (U_j,\Omega).
$

(7) If $E\subset  D \subset \Omega$, then $C  ^*(E,D)\leq C  ^*(E,\Omega)$.
\end{prop}
\begin{proof} (1) For any  $u\in \mathcal{U}^*(K,\Omega)$  in \eqref{eq:U*} and any $0<\varepsilon <1$, consider the open set
 \begin{equation*}
    O:= \left\{ \mathbf{x}\in \Omega; u( \mathbf{x})<(1-\varepsilon)\omega^*( \mathbf{x},K,\Omega)-\varepsilon/2 \right\}.
 \end{equation*}
 Since $\omega^* (\mathbf{x},K,\Omega)=-1$ for $\mathbf{x}\in K$ and is continuous by Proposition \ref{prop:maximal}, we have   $   O\supset K$. Here
 $-\varepsilon/2 $  promises that $O$ is relatively compact in $\Omega$, i.e.  $ \Omega\Supset O\supset K$.  Then
  apply the  comparison principle  to get
\begin{equation*}\begin{split}(1-\varepsilon)^2\int_{K}\det(Hess_{\mathbb{O}}(    \omega^*( \cdot,K,\Omega)))=
&  (1-\varepsilon)^2\int_{O}\det(Hess_{\mathbb{O}}(    \omega^*( \cdot,K,\Omega)))\\
   \leq&  \int_{O}\det(Hess_{\mathbb{O}}(    u))\leq \int_{\Omega}\det(Hess_{\mathbb{O}}(    u))
 \end{split}\end{equation*}
by   \eqref{eq:omega=0}. Letting $\varepsilon\rightarrow0$, we see that  the infimum on the R. H. S. of \eqref{capacity defi-K} is attained by
 $ \omega^*( \mathbf{x},K,\Omega) $.

(2) $C  (K)\leq C  (E)$ by monotonicity. Conversely, for   any $0<\varepsilon <1$, choose  $u\in\mathcal{U}^*(K,\Omega)$ such that $
\int_{\Omega}\det(Hess_{\mathbb{O}}(u  ) ) <C
(K)+\varepsilon$. Since $U:=\{\mathbf{x}\in\Omega; u(\mathbf{x})<-1+\varepsilon/2\}$ is  a neighborhood of the compact set $K$, there exists a  regular
compact subset $E$    such that $
K\subset  E\Subset  U$ by Corollary \ref{cor:regular}.  Consider
\begin{equation*}
   O:=\left\{\mathbf{x}\in \Omega; u(\mathbf{x})<(1-  \varepsilon)\omega^*(\mathbf{x},E,\Omega)-\varepsilon/2 \right\}  .
\end{equation*}Then, $ E\subset O\Subset  \Omega$ as above,
 and so
\begin{equation*}\begin{split}C  ( E)&= \int_{E}\det(Hess_{\mathbb{O}}(  \omega^*(\mathbf{x},E,\Omega)))\leq \int_{O}\det(Hess_{\mathbb{O}}(
\omega^*(\mathbf{x},E,\Omega)))\\
&\leq\frac 1{ (1- \varepsilon)^{2}}\int_{O}\det(Hess_{\mathbb{O}}(u))
    \leq \frac 1{ (1-  \varepsilon)^{2}}\int_{\Omega}\det(Hess_{\mathbb{O}}(u)) \leq \frac {C  (K)+\varepsilon}{ (1- \varepsilon)^{2}},
 \end{split}\end{equation*}by using  \eqref{eq:capacity-K} for    regular  compact subset $E$ and the  comparison principle. The result follows by
 letting
 $\varepsilon\rightarrow0$.

(3) $ C  (K)$ is less than or equal to  the right hand side of  \eqref{eq:capacity-K-2} by using  \eqref{eq:capacity-K}. On the other hand, for any $u_j\in
OPSH(\Omega)\cap
C(\Omega)$ with $ -1\leq u_j< 0$, consider
\begin{equation*}
   v_j( \mathbf{x}):=\max\left\{(1+\varepsilon)\omega^*( \mathbf{x},K,\Omega),(1-\varepsilon)u_j( \mathbf{x})-\varepsilon/2 \right\}.
\end{equation*}
  Then,
$v_j\in OPSH(\Omega)\cap C(\Omega)$ with $ -1\leq v_j< 0$, $\lim_{\mathbf{x}\rightarrow\partial\Omega}v_j( \mathbf{x})=0$, and
$v_j\equiv(1+\varepsilon)\omega^*(\cdot,K,\Omega)$ near the
boundary. We get
\begin{equation*}\begin{split} (1+\varepsilon)^{ 2}\int_{\Omega}\det(Hess_{\mathbb{O}}( \omega^*( \cdot,K,\Omega))) &= \int_{\Omega} \det(Hess_{\mathbb{O}}(v_1 ),
Hess_{\mathbb{O}}(v_2 )) \\&  \geq
(1-\varepsilon)^2 \int_{K} \det(Hess_{\mathbb{O}}(u_1 ), Hess_{\mathbb{O}}(u_2 )) .
 \end{split}\end{equation*}by using Corollary \ref{cor:compare} and $v_j\equiv(1-\varepsilon)u_j -\varepsilon/2 $ on $K$. Letting
 $\varepsilon\rightarrow0$, we
 get the another direction of inequality,
  since $\det(Hess_{\mathbb{O}} (\omega^* ))  =0$ on $\Omega\setminus  K $.

  (4) For any $u_1,u_2\in OPSH(\Omega)\cap C(\Omega)$ with $-1\leq  u_1,u_2<0$, we have
    \begin{equation*}
     C  (U) \geq C  (K)\geq  \int_{K}\det(Hess_{\mathbb{O}}(u_1), Hess_{\mathbb{O}}(u_2))
  \end{equation*} by (3). Then $C  (U)  \geq
  \int_{U}\det(Hess_{\mathbb{O}}(u_1), Hess_{\mathbb{O}}(u_2))$,
   since $K$ can be arbitrarily chosen  regular compact subset. Thus $C  (U)$ is larger than or equal to the R. H. S. of \eqref{capacity defi}.

On the other hand, since $\Omega$ is a strongly    hyperconvex domain, the relative    extremal function $ \omega^* (\mathbf{x} ,K,\Omega)$ admits an   $OPSH$
extension to a neighborhood of   $\overline{\Omega}$ by Proposition \ref{prop:relative-extremal} (4), and so it can be approximated in a neighborhood   of
$\overline{\Omega}$ by $OPSH \cap C^\infty $
  functions $\widetilde{v}_{j}\downarrow  \omega^* ( \cdot,K,\Omega)$. Hence, by Theorem A,
  \begin{equation*}\begin{split}C  ( K)&= \int_{K} \det(Hess_{\mathbb{O}}(   \omega^*( \cdot, K ,\Omega))) = \int_{\Omega}\det(Hess_{\mathbb{O}}(   \omega^*(
  \cdot, K ,\Omega) ))\\
&\leq \overline{\lim} _{j\rightarrow \infty} \int_{\Omega}\det(Hess_{\mathbb{O}}(\widetilde{ v}_{j}))   = \overline{\lim}_{j\rightarrow \infty} (1-\varepsilon)^{
-2}\int_{\Omega}\det(Hess_{\mathbb{O}}(\widetilde{ w}_{j} ))
 \end{split}\end{equation*}
if we denote $ \widetilde{w}_{j}=(1-\varepsilon)  \widetilde{v}_{j}-\varepsilon $. Here $-1\leq \widetilde{w}_{j}<0$  if $j$  is large. So $C  ( K)$ is controlled by the right hand side of
\eqref{capacity defi}
multiplying $(1-\varepsilon)^{
-2}$. Since $K$ is an arbitrarily chosen compact subset, the result follows by letting $\varepsilon\rightarrow0$.

  (5) The monotonicity of $C  ^*(E ) $ follows from the monotonicity of $C   ( K)$  for compact subsets $K$. If $E_j$'s are open sets, then by using (4), we have
  \begin{equation*}\begin{split}
C^*  ( \cup_j E_j)&=
 \sup_{u \in  \mathscr  L} \int_{\bigcup_j E_j}\det(Hess_{\mathbb{O}}(u))   \leq\sup_{u \in  \mathscr  L}  \sum_j\int_{ E_j}\det(Hess_{\mathbb{O}}(u))   \leq
 \sum_j C^*  (  E_j).
 \end{split}\end{equation*}
  In general, we can find an open set $U_j\supset E_j$ such that $C  ( U_j) -C  ^*(  E_j)\leq \varepsilon/2^j$ for each $j$. Then
  \begin{equation*}
  \varepsilon+   \sum_j C  ^*(  E_j)\geq \sum_j C  ( U_j)  \geq C  ( \cup_j U_j) \geq C  ( \cup_j E_j) .
  \end{equation*}
  We get the result by letting $\varepsilon\rightarrow0$.

   (6)-(7) They are  obvious by definition.
 \end{proof}

By (4) and (5), we get a useful  estimate: for  a octonionic  strongly hyperconvex domain $\Omega$, there exists a neighborhood $\Omega'\supset\overline{\Omega}$
such that
\begin{equation}\label{eq:outside}
\int_{U}\det(Hess_{\mathbb{O}}(u_1), Hess_{\mathbb{O}}(u_2))  \leq C  (U)
\end{equation} for any $u_j \in OPSH(\Omega')\cap C(\Omega')$ with $ -1\leq u_j < 0$ on $\Omega$ and $|u_j |\leq1$ on $\Omega'$, $j=1,2$.
 \begin{rem}
Several results about relative extremal functions  and     capacity  in Sadullaev-Abdullaev \cite{sadullaev1,sadullaev2} can also be
    proved by only
using Chern-Levine-Nirenberg inequalities for complex $m$-Hessian (cf. Nguyen \cite{Nguyen-C}). This method  does
not depend on the solution to the
Dirichlet problem  on the ball. So does it in the  octonionic  case.
\end{rem}

 \section{ The quasicontinuity of  locally  bounded    $OPSH$ functions}
 
\begin{lem}\label{lem: inner-product} \cite[Corollary 3.1]{WZ}
If $u,v\in C^2(\Omega)$ and let $ \omega$ be a nonnegative  continuous $\mathcal{ H}^2(\mathbb{O} )$-valued function. Then
\begin{equation}
\begin{aligned}
\left|\int_{\Omega}  \det\left( \mathcal{T} (d u \otimes d v),\omega  \right)\right|^2\leq  \int_{\Omega}\det\left( \mathcal{T} (d u \otimes d u),\omega
\right)\cdot
\int_{\Omega}\det\left( \mathcal{T} (d v \otimes d v),\omega  \right).
\end{aligned}
\end{equation}
\end{lem}
\begin{proof} Note that the map $(u, v) \rightarrow \det\left( \mathcal{T} (d u \otimes d v),\omega) \right)$ is real bilinear and symmetric, and by Proposition \ref{prop:positive}  (1) and Proposition \ref{prop:det-positive} (1),  $\det\left(
\mathcal{T} (d u \otimes d
u),\omega  \right)$  is
 nonnegative. Therefore  it
 satisfies
the Cauchy-Schwarz inequality.\end{proof}

\begin{proof}[Proof of Theorem \ref{thm:quasicontinuity}] (1) Since the capacity is countably subadditive by Proposition \ref{prop:capacity} (5), it is sufficient  to prove the theorem for the unit ball
$B^2\subset \Omega$ and show
that for any $\epsilon >0$ there exists an open set $U\subset B^2 $ such $C  (U\cap B',B)<\epsilon$ and $u$ is continuous on $B'\setminus U,$ where
$B'=B(\mathbf{0}, {1}/{2})$. Assume
$-1\leq   u\leq   0$.  As in the proof of  Proposition
\ref{prop:uniformly-bounded}, if replace $u$ by $\max\{u ,v \}$ with $v(\mathbf{x})=2(|\mathbf{x}|^2-\frac{3}{4}),$ then values inside $B'$
are unchanged, while   $u\equiv v$ in a neighborhood of the sphere $ \partial B^2$. Let  $u_p \downarrow u$, $v_p\downarrow  v$ be the standard
approximations. Note that
$u_p\equiv v_p$ in  a neighborhood of $ \partial B$ for $p>p_0$.
We can assume
\begin{equation}\label{eq:lim}
 \lim_{p\rightarrow +\infty}  \int_B u_p\det(Hess_{\mathbb{O}}(u_p)   )
\end{equation}
 exists by passing to a subsequence if necessary, since they are  bounded by  Proposition
\ref{prop:uniformly-bounded}.  For a fixed $\sigma>0$, consider
\begin{equation*}
   U_{p,N}(\sigma):=\{\mathbf{x}\in B' : \varphi_{p,N}(\mathbf{x}) := u_p(\mathbf{x})-u_{p+N}(\mathbf{x}) >\sigma\},
\end{equation*}
  then we have $
 U_{p,N}(\sigma)\subset U_{p,N+1}(\sigma),$ and
$\bigcup_{N=1}^\infty U_{p,N}=U_p(\sigma):=\{\mathbf{x}\in B':u_p(\mathbf{x})-u(\mathbf{x})>\sigma\}.
$
Then we have
 $ C  (U_p(\sigma))=
  C  \left(\bigcup_{N=1}^\infty U_{p,N}(\sigma)\right)=\lim\limits_{N\to\infty}C  (U_{p,N}(\sigma))
$
by    Proposition \ref{prop:capacity} (6).

Consider the class $\mathcal{L}:=\{u\in OPSH\cap C^\infty(B(\mathbf{0}, 1 + \delta))  ;|u| \leq 1\}$   for fixed  $\delta> 0$.  Since the open set $
U_{p,N}(\sigma)\subset B'\Subset {  B^2} $, it follows from
  Proposition \ref{prop:capacity} (4)  that
  \begin{equation}\label{eq:cap-p-N}
 \begin{split}
C  (U_{p,N}(\sigma))&\leq\sup_{\widetilde{u}\in \mathcal{L}} \int_{U_{p,N}(\sigma)}\det(Hess_{\mathbb{O}}(\widetilde{u})  )
   \leq   \sup_{\widetilde{u}\in \mathcal{L}}   \frac{1}{\sigma}\int_{B^2}\varphi_{p,N}\det(Hess_{\mathbb{O}}(\widetilde{u})  ) .
  \end{split}
  \end{equation}

(2) To estimate the R. H. S. of \eqref{eq:cap-p-N}, we establish first  an integral estimate for   functions  in $\mathcal{L}$. Let
$\varrho:=(|\mathbf{x}|^2-1)/2$ be a defining function of the unit ball
${  B^2} $.
Consider functions $\widetilde{  v},  \widetilde{u},  \widetilde{ u}_1,   \widetilde{ u}_2\in \mathcal{L} $ such that $\varphi_0= \widetilde{ v}-  \widetilde{
u}\geq  0$ in $B^2$ and $\varphi_0=$const. on the sphere $ \partial {  B^2} $. Then
\begin{equation}
\begin{aligned} &
\int_{  B^2}  \varphi_0\det(Hess_{\mathbb{O}}(\widetilde{ u}_1) ,Hess_{\mathbb{O}}(\widetilde{  u}_2)) \\
 =& \varphi_0|_{\partial {  B^2}}\int_{\partial B^2} \det\left(   \mathcal{T} (d \widetilde{ u}_1 \otimes d \varrho) ,Hess_{\mathbb{O}}( \widetilde{ u}_2) \right) dS-\int_{  B^2}
 \det\left( \mathcal{T} (d \varphi_0 \otimes d \widetilde{  u}_1 )  ,Hess_{\mathbb{O}}(\widetilde{  u}_2) \right)
   \\=&\varphi_0 |_{\partial {  B^2}}\int_{  B^2}    \det\left(  Hess_{\mathbb{O}}( \widetilde{ u}_1) ,Hess_{\mathbb{O}}( \widetilde{ u}_2) \right) -\int_{  B^2}
 \det\left( \mathcal{T} (d \varphi_0 \otimes d  \widetilde{ u}_1 )  ,Hess_{\mathbb{O}}( \widetilde{ u}_2) \right) \\
\leq &  C \|\varphi_0\|_{C(\partial {  B^2} )}- \int_ {  B^2}
 \det\left( \mathcal{T} (d \varphi_0 \otimes d \widetilde{  u}_1 )  ,Hess_{\mathbb{O}}(\widetilde{  u}_2) \right)  
\end{aligned}
\end{equation}by using integration by parts in Lemma \ref{lem:parts}   and $|{\rm grad} \varrho|=1$ on $\partial B^2$,
  where $ C $ is  an absolute constant  independent of $ \widetilde{ u}_1, \widetilde{  u}_2\in \mathcal{L} $ by  Proposition  \ref{prop:uniformly-bounded}.
Applying Lemma  \ref{lem: inner-product}    to $u=\varphi_0, v=\widetilde{  u}_1$ and smooth closed nonnegative function $\omega= Hess_{\mathbb{O}}(\widetilde{
u}_2)  $, and using integration by parts \eqref{lem:parts}
twice, we get
\begin{equation*}\label{eq:phi0}
\begin{aligned}
\left|\int_ {  B^2}  \det\left( \mathcal{T} (d \varphi_0 \otimes d  \widetilde{u}_1 )  ,\omega \right) \right|^2&\leq   \int_{  B^2}  \det\left( \mathcal{T} (d
\widetilde{ u}_1 \otimes d \widetilde{ u}_1 )  ,\omega
\right) \int_{  B^2}  \det\left( \mathcal{T} (d \varphi_0 \otimes d \varphi_0)  ,\omega \right) \\
 &\leq      C \left(\varphi_0|_{\partial {  B^2}} \int_{\partial {  B^2} }\det\left( \mathcal{T} (d \varphi_0 \otimes d\varrho  ,\omega \right)dS
 -\int_{  B^2} \varphi_0 \det(Hess_{\mathbb{O}}(\varphi_0) ,\omega)\right) \\&=
 C \left(\varphi_0|_{\partial {  B^2}}\int_{  B^2}  \det(Hess_{\mathbb{O}}(\varphi_0) ,\omega)-\int_{  B^2}  \varphi_0\det(Hess_{\mathbb{O}}(\varphi_0) ,\omega)\right)  \\
&\leq   C    \left(2C \|\varphi_0\|_{C(\partial{  B^2} )}+\int_B 2 \varphi_0\det\left(Hess_{\mathbb{O}}\left(\frac{\widetilde{  u}+ \widetilde{ v}}{2}\right)
-Hess_{\mathbb{O}}(  \widetilde{v})   ,\omega)  \right)   \right)\\
&\leq  2C \left( C \|\varphi_0\|_{C(\partial {  B^2} )}+ \int_{  B^2}  \varphi_0\det(Hess_{\mathbb{O}}(\varphi_0^{+}) ,\omega)  \right),
\end{aligned}
\end{equation*}
where $\varphi_0^{+}=\frac{ \widetilde{ u}+ \widetilde{v}}{2}\in \mathcal{L}$. The   second inequality follows from locally uniform estimate in  Proposition
\ref{prop:uniformly-bounded}, and
$\varphi_0\vert_{ \partial {  B^2}  }=\|\varphi_0\|_{C(\partial {  B^2} )}$,  and the third identity holds since
$$ \left|\int_{  B^2}  \det(Hess_{\mathbb{O}}(\varphi_0) ,\omega)\right|\leq    \int_{  B^2}   \det(Hess_{\mathbb{O}}( \widetilde{ u}+ \widetilde{ v} )
,\omega)    \leq
2C   ,$$
while the last inequality above follows from the fact $\varphi_0\geq  0$ and $Hess_{\mathbb{O}}( \widetilde{ u}+ \widetilde{ v }),\omega=
Hess_{\mathbb{O}}(  \widetilde{u}_2)  \geq  0$.

Applying this procedure again, we obtain
\begin{equation}\label{eq:kappa}
\begin{aligned}
\int_{  B^2} \varphi_0 \det(Hess_{\mathbb{O}}( \widetilde{ u}_1) ,Hess_{\mathbb{O}}( \widetilde{ u}_2)) \leq   C''\left(\|\varphi_0\|_{C(\partial B^2)}+\int_{  B^2}
\varphi_0\det(Hess_{\mathbb{O}}(\varphi_0^{+})  ) \right)^\frac 14,
\end{aligned}
\end{equation}
for some absolute constant  $C''>0$.

 (3) Apply the estimate \eqref{eq:kappa} to \eqref{eq:cap-p-N} with
$ \varphi_0 =u_p-u_{p+N}$ to get
   \begin{equation}\label{eq:capacity-estimate}
  \begin{aligned}
C  (U_{p,N}(\sigma))&
&\leq  \frac{C''}{\sigma}\left(\|v_p-v_{p+N}\|_{C(\partial {  B^2} )}+\int_{  B^2}  \varphi_{p,N}\det\left(Hess_{\mathbb{O}}\left( \varphi_{p,N}^{+} \right)   \right)\right)^\frac 14,\\
\end{aligned}
  \end{equation}
    by  $ u_p-u_{p+N} = v_p-v_{p+N} $   constant on $\partial {  B^2} $,
where $\varphi_{p,N}^+:=(u_p+u_{p+N})/2.$ Note that
   \begin{equation}\label{eq:capacity-estimate'}
  \begin{aligned}
4\det(Hess_{\mathbb{O}}( \varphi_{p,N}^{+} )   )& = \det(Hess_{\mathbb{O}}(u_p+ u_{p+N}  )   ) \\&  =\det(Hess_{\mathbb{O}}(u_p   )   )+ 2
\det(Hess_{\mathbb{O}}(u_p),
Hess_{\mathbb{O}}(  u_{p+N}  )       )+\det(Hess_{\mathbb{O}}(  u_{p+N}  )   ) .
\end{aligned}
  \end{equation}

Let us  prove
\begin{equation}\label{eq:difference-limit-1}
\int_{  B^2} (u_p-u_{p+N})\det(Hess_{\mathbb{O}}(u_p),  Hess_{\mathbb{O}}(  u_{p+N}  )      )\rightarrow 0,
\end{equation}
    uniformly  as $N\to\infty $ and then
  $p\to\infty $.
 If denote $\omega=Hess_{\mathbb{O}}(u_p) $, we have
   \begin{equation*}
  \begin{aligned}
\int_{  B^2}  u_p\det (  Hess_{\mathbb{O}}(  u_{p+N} ) ,\omega )
= &\int_{\partial {  B^2} }   u_p \det\left(   \mathcal{T} (d u_{p+N} \otimes d \varrho) ,\omega \right) dS -\int_ {  B^2}
 \det\left( \mathcal{T} (d u_p  \otimes d u_{p+N} )  ,\omega\right)
  \\
= &\int_{\partial {  B^2} }   u_p \det\left(   \mathcal{T} (d u_{p+N} \otimes d \varrho),\omega\right) dS -\int_{\partial {  B^2} } u_{p+N} \det\left(
\mathcal{T} (d u_p \otimes d \varrho) ,\omega \right) dS
\\
& +\int_ {  B^2}  u_{p+N}
 \det\left( Hess_{\mathbb{O}}(  u_{p } ) ,\omega\right)\\
\leq &   A_{p,N}+\int_{B^2} u_p \det\left( Hess_{\mathbb{O}}(  u_{p } ) ,Hess_{\mathbb{O}}(u_p)\right),
\end{aligned}
  \end{equation*} by using integration by parts in \eqref{lem:parts},
  where
  \begin{equation}
    A_{p,N}:=\int_{\partial {  B^2} } v_p \det\left(   \mathcal{T} (d v_{p+N} \otimes d \varrho)  ,Hess_{\mathbb{O}}(v_p) \right)dS -\int_{\partial {  B^2} }
    v_{p+N} \det\left(   \mathcal{T} (d v_p \otimes d \varrho)
    ,Hess_{\mathbb{O}}(v_p) \right)  dS,
  \end{equation}since $u_p=v_p$ in a neighborhood of $\partial {  B^2} $ for   $p>p_0$.
  Similarly, by exchanging $p$ and $p+N$, we get
     \begin{equation}
  \begin{aligned}
 \int_{  B^2}  u_{p+N}\det (  Hess_{\mathbb{O}}(  u_{p } ) , Hess_{\mathbb{O}}(  u_{p+N}  )   )    &=B_{p,N}+\int_{  B^2}  u_p \det (  Hess_{\mathbb{O}}(  u_{p+N}
 ) , Hess_{\mathbb{O}}(  u_{p+N}  )   )
\\& \geq  B_{p,N}+\int_{  B^2}  u_{p+N}\det (  Hess_{\mathbb{O}}(  u_{p+N} ) , Hess_{\mathbb{O}}(  u_{p+N}  )  ) ,
\end{aligned}
  \end{equation}
by $u_p\geq u_{p+N}$ and $  Hess_{\mathbb{O}}(  u_{p+N} )   \geq 0$,  where
  \begin{equation*}
     B_{p,N}:=\int_{\partial {  B^2} } v_{p+N} \det\left(   \mathcal{T} (d v_p \otimes d \varrho) , Hess_{\mathbb{O}}(  v_{p+N}  )  \right)dS -\int_{\partial {
     B^2} }  v_p \det\left(   \mathcal{T} (d v_{p+N}\otimes d \varrho)
     ,Hess_{\mathbb{O}}(  v_{p+N}  ) \right)dS .
  \end{equation*}

   Finally we get
     \begin{equation}\label{eq:difference-pN}
  \begin{aligned}
 &\int_{  B^2} (u_p-u_{p+N})\det(Hess_{\mathbb{O}}(u_p),  Hess_{\mathbb{O}}(  u_{p+N}  )   )   ) \\
\leq &  A_{p,N}-B_{p,N} +\int_{  B^2}  u_p\det(Hess_{\mathbb{O}}(u_p)   )-\int_B u_{p+N} \det   Hess_{\mathbb{O}}(  u_{p+N}  )     ).
\end{aligned}
  \end{equation}
  Because the sequence $\{v_p\}$ converges in   $C^2(\overline{  B^2} )$, we have
    \begin{equation}
   \begin{split}
 A_{p,N}&\rightarrow \int_{\partial {  B^2} }  v_p \det\left(   \mathcal{T} (d v  \otimes d \varrho)  ,Hess_{\mathbb{O}}(v_p) \right) -\int_{\partial {  B^2} } v
 \det\left(   \mathcal{T} (d v_p \otimes d \varrho)
 ,Hess_{\mathbb{O}}(v_p) \right)
\rightarrow0,
 \end{split}
  \end{equation}
as $N\to\infty$ and then $p\rightarrow \infty$. Similarly  $B_{p,N}\rightarrow 0$.
Since the sequence $\int_{  B^2}  u_p\det(Hess_{\mathbb{O}}(u_p)   )$ in \eqref{eq:lim} has a limit as  $p\rightarrow \infty$,  the right hand side of
\eqref{eq:difference-pN} tends to $0$, i.e. \eqref{eq:difference-limit-1} holds.
Similarly, we can show
\begin{equation}\label{eq:difference-limit-2}  \begin{split}&
   \int_{  B^2} (u_p-u_{p+N})\det(Hess_{\mathbb{O}}(u_p),  Hess_{\mathbb{O}}(  u_{p }  )     )\rightarrow 0, \\
    &\int_{  B^2} (u_p-u_{p+N})\det(Hess_{\mathbb{O}}( u_{p+N}),  Hess_{\mathbb{O}}(  u_{p+N}  )      )\rightarrow 0,
 \end{split}\end{equation}  as $N\to\infty $ and then
  $p\to\infty $.
 Apply \eqref{eq:difference-limit-1} and \eqref{eq:difference-limit-2} to \eqref{eq:capacity-estimate}-\eqref{eq:capacity-estimate'} to get
$$\lim\limits_{p\to\infty}C  (U_p(\sigma))=\lim\limits_{p\to\infty}\lim\limits_{N\to\infty}(U_{p,N}(\sigma))=0.$$

 Now   fixed $\epsilon>0$, for $\sigma=\frac{1}{j}$, there exists $p_j>0$ such that if we denote $U_{p_j}:=U_{p_j}( {1}/{j})$,   we have $C  (U_{p_j})\leq
 \frac{\epsilon}{2^j}$.
 Since $u_p(\mathbf{x})-u(\mathbf{x})<\frac{1}{j}$ for $p>p_j$ outside the set
 $U_{p_j}$, then we see that $u_p$ convergence to $u$ uniformly outside the open set $U=\cup_{j=1}^{\infty}U_{p_j}$. Since $u_p\in C^{\infty}({  B^2} ),$   $u$ is
 continuous outside $U$,
 and
  $C  (U )=C  \left(\bigcup_{j=1}^\infty U_{p_j}\right)\leq  \sum\limits_{j=1}^\infty C  (U_{p_j})\leq   \epsilon .$
 The theorem is proved. \end{proof}

 \appendix

\section{Automorphisms of the unit ball of $\mathbb{O}^2$  }

The octonionic Heisenberg group $\mathscr{H}$ is $\mathbb{O}\oplus{\rm Im }\ \mathbb{O}$ equipped
with the multiplication given by
\begin{align}\label{heiO}
(\mathbf{x},\mathbf{t})\cdot(\mathbf{x}',\mathbf{t}')=\left(\mathbf{x}+\mathbf{x}',\mathbf{t}+\mathbf{t}'+2{\rm{Im}} (\overline{ \mathbf{x} } \mathbf{x}' )\right),
\end{align}
where $(\mathbf{x},\mathbf{t}),(\mathbf{x}',\mathbf{t}')\in\mathbb{O}\oplus{\rm Im }\ \mathbb{O}.$
We have the following transformations   of $\mathscr{H}$: 
(1) \emph{dilations}:  $
D_{\delta} (\mathbf{x},\mathbf{t})=(\delta \mathbf{x},\delta^{2}\mathbf{t} ), 
$ for given  positive number $\delta $; 
(2) \emph{left translations}: $
\tau_{(\mathbf{x}',\mathbf{t}')}(\mathbf{x},\mathbf{t})=(\mathbf{x}',\mathbf{t}')\cdot(\mathbf{x},\mathbf{t})
$,   for given $(\mathbf{x}',\mathbf{t}')\in \mathscr H$;  
(3) \emph{rotations}:  $
S_\mu(\mathbf{x},\mathbf{t})=({\mu}\mathbf{x},  \mu  \mathbf{t} \bar{\mu})
$, for given unit imaginary octonion $\mu$;
(4) an \emph{inversion}.
 It is known that
these transformations generate the exceptional group ${\rm F}_{4(-20)}$. They can be naturally extended to automorphisms of the octonionic  Siegel upper half space
\eqref{eq:Siegel} (cf. e.g. \cite{All,SW}). In particular,
the left translate
can be extended to a mapping  $\tau_\zeta:\mathcal U\rightarrow \mathcal U$ for $\zeta\in\partial \mathcal U $ given by \eqref{eqLeft-translate}. This is  because
\begin{equation*}\begin{split}
2\operatorname{ Re} {\mathbf{y}}_2
 -|  { \mathbf{y} }_1|^2&= 2\operatorname{ Re} \mathbf{x}_2 +2\operatorname{ Re} \zeta_2 +2\operatorname{ Re}(\overline{\zeta}_1 \mathbf{x}_1
 )-|\mathbf{x}_1+\zeta_1|^2
 =2\operatorname{ Re} \mathbf{x}_2-|   \mathbf{x  }_1|^2 ,
\end{split}\end{equation*} which implies  $\tau_\zeta$ mappings the boundary $ \partial \mathcal U $ to itself.
The octonionic Heisenberg group can be identified with the boundary  $\partial\mathcal{ U}$ of the octonionic  Siegel upper half space.
Let
\begin{equation}\label{eq:delta-a}
   \delta _{\mathbf{a}}:=\frac {|1+\mathbf{a}_{2}|}{(1-| {\mathbf{a}} |^2 )^\frac 12} .
\end{equation}and
let $ \zeta_{\mathbf{a}}$ be the projection of $ \overline{C(\mathbf{a})}$ to $\partial \mathcal U $, i.e.
\begin{equation}\label{eq:zeta-a}
   \zeta_{\mathbf{a}}:= \left (-  [C(\mathbf{a})] _1, \frac 12 |[C(\mathbf{a}) ]_1|^2-\operatorname{ Im}[C(\mathbf{a})]_2\right)=\left(-\sqrt 2
   \mathbf{a}_{1}(1+\mathbf{a}_{2})^{-1} ,\frac { |\mathbf{a}_{1}|^2 }{|1+\mathbf{a}_{2}|^2}+\frac {  2\operatorname{ Im} \mathbf{a}_{2}   }{|1+\mathbf{a}_{2}|^2}
  \right).
\end{equation}
It is direct to check that
\begin{equation}\label{eq:zeta-a-1}
 \tau_{\zeta_{\mathbf{a}}}( C(\mathbf{a}))=  \left(0, \frac {1-| \mathbf{{a}} |^2 }{|1+\mathbf{a}_{2}|^2} \right),
\end{equation}by \eqref{eq:Siegel-diff}, and so  $D_{\delta _{\mathbf{a}}}  \circ  \tau_{ \zeta_{\mathbf{a}}}\circ  C(\mathbf{a})=(0,1 )$. Thus,  we have
$
   T_{\mathbf{a}}(\mathbf{a})=\mathbf{0}  
$ by \eqref{eq:C0}.

Since for  $ \mathbf{a} \in B^2$, $D_{\delta _{\mathbf{a}}} $ and $  \tau_{ \zeta_{\mathbf{a}}} $ are diffeomorphisms of  the octonionic  Siegel upper half space
 $ \mathcal U $ to itself and $C$ is a diffeomorphism  from the ball $B^2$  to  $ \mathcal U $,  $T_{\mathbf{a}}$  defined by
\eqref{eq:T--a}  is obviously a diffeomorphism of the unit ball $B^2$ to itself.
By definition \eqref{eq:T--a},
\begin{equation}\label{eq:T-G}\begin{split}
    T_{\mathbf{a}}(\mathbf{x})
&=  C^{-1}    \left(\delta _{\mathbf{a}}\left(\sqrt 2\mathbf{x}_{1}(1+\mathbf{x}_{2})^{-1}+[\zeta_{\mathbf{a}}]_1\right),  G_{\mathbf{a}}(\mathbf{x})-1\right) ,
 \end{split}  \end{equation}
 where
 \begin{equation}\label{eq:Ga}
   G_{\mathbf{a}}(\mathbf{x}): =  1+\delta _{\mathbf{a}}^{ 2} \Big( (1- \mathbf{x}_{2}) (1+\mathbf{x}_{2})^{-1}+   [\zeta_{\mathbf{a}}]_2 +
\overline{[\zeta_{\mathbf{a}}]}_1  \sqrt 2\left(\mathbf{x}_{1}(1+\mathbf{x}_{2})^{-1}\right)\Big)
  .
\end{equation}
We find that
 \begin{equation}\label{eq:T-a}\begin{split}
    [T_{\mathbf{a}}(\mathbf{x})]_1&=\frac 1{\sqrt 2}\delta _{\mathbf{a}}    \left\{\sqrt 2 \mathbf{x}_{1}(1+\mathbf{x}_{2})^{-1} + [\zeta_{\mathbf{a}}]_1   \right\}
    \left[ (1+\mathbf{x}_{2})
 \Psi_{\mathbf{a}}(\mathbf{x}) ^{-1}\right],
\\
    [T_{\mathbf{a}}(\mathbf{x})]_2&=-1+2
    (1+\mathbf{x}_{2})
   \Psi_{\mathbf{a}}(\mathbf{x}) ^{-1},
  \end{split}   \end{equation}by definition of $\Psi_{\mathbf{a}}$ in \eqref{eq:T-a-u}, from which we see that
\begin{equation}\label{eq:Psi-G}
   \Psi_{\mathbf{a}}(\mathbf{x}) = G_{\mathbf{a}}(\mathbf{x})  (1+\mathbf{x}_{2}),
\end{equation}by \eqref{eq:T-G}.
  It is easy to see that
  \begin{equation}\label{eq:T-G2}\begin{split}
   \operatorname{ Re}   G_{\mathbf{a}}(\mathbf{x} )
    &=1+\delta _{\mathbf{a}}^{ 2}\frac {1-|\mathbf{x} |^2  }{|1+\mathbf{x}_{2}|^2}+ \delta _{\mathbf{a}}^{ 2}\left|\mathbf{x}_{1}(1+\mathbf{x}_{2})^{-1}-
    \mathbf{a}_{1}(1+\mathbf{a}_{2})^{-1}\right |^2\geq 1.
  \end{split} \end{equation}

\begin{proof}[Proof of Proposition \ref{prop:Psi-a}]
(1) By the expression  \eqref{eq:Psi-G} of $ \Psi_{\mathbf{a}}(\mathbf{x})$ and \eqref{eq:Ga}, we can write
  \begin{equation}\label{eq:Psi-a}
    \Psi_{\mathbf{a}}(\mathbf{x}) = (1+\delta _{\mathbf{a}}^{ 2})+(1-\delta _{\mathbf{a}}^{ 2})\mathbf{x}_{2}+ \delta _{\mathbf{a}}^{ 2}
   [\zeta_{\mathbf{a}}]_2(1+\mathbf{x}_{2}) +\mathcal{R}(\mathbf{x})
 ,
\end{equation}where
\begin{equation}\label{eq:E}
  \mathcal{R}(\mathbf{x}):= \sqrt 2\delta _{\mathbf{a}}^{
   2}\left ( \overline{[\zeta_{\mathbf{a}}]}_1\left(\mathbf{x}_{1}(1+\mathbf{x}_{2})^{-1}\right)\right)(1+\mathbf{x}_{2}) .
\end{equation} 
For $\mathbf{a}\in \overline{B}(\mathbf{0}, 1-\varepsilon)$, $\zeta_{\mathbf{a}}$ belongs to a bounded subset of  $\partial \mathcal U $ by its expression in
\eqref{eq:zeta-a}, and then it is continuous in $\mathbf{a}$. So is $\delta_{\mathbf{a}}$. $ \Psi_{\mathbf{a}}(\mathbf{x})$ is obviously smooth for  $\mathbf{a}, \mathbf{x}\in \overline{B}(\mathbf{0}, 1-\varepsilon)$ by  expressions in \eqref{eq:Psi-a}-\eqref{eq:E}.

(2)
Set $\mathcal{R}(0,-1)=0$. Note that
\begin{equation*}
  \frac {\partial \mathcal{ R}}{\partial x_{2p}}= -\delta _{\mathbf{a}}^{
   2}\left ( \overline{[\zeta_{\mathbf{a}}]}_1\left(\mathbf{x}_{1}((1+\mathbf{x}_{2})^{-2}e_p)\right)\right)(1+\mathbf{x}_{2}) +\delta _{\mathbf{a}}^{
   2}\left ( \overline{[\zeta_{\mathbf{a}}]}_1\left(\mathbf{x}_{1}(1+\mathbf{x}_{2})^{-1}\right)\right)e_p
\end{equation*}is uniformly bounded by $  C| 1-\mathbf{x}_2| ^{-\frac 12}$, since $|\mathbf{x}_1|\leq( 1-|\mathbf{x}_2|^2)^\frac 12$, and so does the derivative with respect to $\partial x_{1p}$.
Thus $ \mathcal{ R}$ is $Lip^{\frac 12}$. In particular, the term $ \mathcal{ R}$ can be extended a continuous function on $\overline{B^2}$ for $\mathbf{a}\in B(\mathbf{0}, 1-\varepsilon)$.

Note that if  $|1+\mathbf{x}_{2}|< \eta_0 <1$, then $| \mathbf{x}_{1}|^2\leq 1-|\mathbf{x}_2|^2\leq
2(1-|\mathbf{x}_2|)\leq 2\eta_0$. Thus if $\eta_0$ is sufficiently small,
we see that $\mathcal{R}(\mathbf{x})$ is small and so $|\Psi_{\mathbf{a}}(\mathbf{x}) |>\delta _{\mathbf{a}}^{ 2}$ by \eqref{eq:Psi-a}-\eqref{eq:E}. While if
$|1+\mathbf{x}_{2}|\geq\eta_0 $, $ \Psi_{\mathbf{a}}(\mathbf{x})$ is obviously smooth. So it has a lower bound independent of $\mathbf{a}\in B(\mathbf{0},
1-\varepsilon) $.

(3) $\delta _{\mathbf{a}}$ and $\zeta_{\mathbf{a}}$ is smooth in $\mathbf{a}\in B(\mathbf{0}, 1-\varepsilon)$ by definitions \eqref{eq:delta-a}-\eqref{eq:zeta-a}.
So  for fixed $\mathbf{x}\in \partial B^2$, $\Psi_{\mathbf{a}}(\mathbf{x})$ are smooth on $\mathbf{a}\in B(\mathbf{0}, 1-\varepsilon)$  by definitions
\eqref{eq:Psi-a}-\eqref{eq:E}.
\end{proof}

 The Corollary \ref{cor:OPSH-Ta} follows from the expression \eqref{eq:T-a} of $T_{\mathbf{a}}(\mathbf{x})$ and Proposition \ref{prop:Psi-a}.
    \begin{rem} \label{rem:smooth} The term $\mathcal{ R}(\mathbf{x})$ in \eqref{eq:E} is only $Lip^{\frac 12}$ on $\overline{B^2}$. This is  because if we write $\mathbf{x}=(
    \mathbf{ c},\mathbf{ d} )$ for two quaternionic numbers $   \mathbf{ c},\mathbf{ d}  $, then by multiplication law \eqref{eq:octonionic} of octonions, $(  \mathbf{ c},\mathbf{ d} )^{-1}=(  \overline{\mathbf{c}},- \mathbf{ d}   )$ and
\begin{equation*}\begin{split}\{(0,\alpha ) [ (0,r)(  \mathbf{ c},\mathbf{ d} )^{-1}]\}(  { \mathbf{ c}}, \mathbf{ d} ) &=r\{(0,\alpha ) (  \overline{\mathbf{d}}, \mathbf{ c}   ) \}(  { \mathbf{ c}}, \mathbf{ d} ) \frac 1{|\mathbf{c}|^2+|\mathbf{d}|^2} \\
  &=r (- \overline{{\mathbf{ c}}} \alpha , \alpha {\mathbf{d}}   )  (   \mathbf{ c} , \mathbf{ d} )\frac 1{|\mathbf{c}|^2+|\mathbf{d}|^2}\\& =r(- \overline{\mathbf{ c}} \alpha
{ \mathbf{ c}}-\overline{\mathbf{d}}\alpha {\mathbf{d}} , -{\mathbf{d}} \overline{ {\mathbf{ c}}} \alpha +\alpha {\mathbf{d}}    \overline{ \mathbf{ c} } )\frac
  1{|\mathbf{c}|^2+|\mathbf{d}|^2}
 \end{split}  \end{equation*}for $r>0$, $\alpha\in\mathbb{ H}$. It does not have a limit at $(  \mathbf{ c},\mathbf{ d} )= (  \mathbf{ 0},\mathbf{ 0} )$ for $\alpha=\mathbf{j}$. 
 If take $\mathbf{x}_{1}= (0,r), 1+\mathbf{x}_{2}=(  \mathbf{ c},\mathbf{ d} ) $ in \eqref{eq:E} such that $r=(1-| \mathbf{x}_{2}|^2)^\frac 12$, we see that
  $\mathcal{ R}$ can only be $Lip^{\frac 12}$ at 
$(0,-1) $. 
 \end{rem}

\end{document}